\pdfoutput=1
\documentclass[dissertation]{aaltoseries}

\usepackage{lipsum}

\usepackage{latexsym}
\usepackage{amsmath}
\usepackage{amsthm}
\usepackage{amssymb}
\usepackage{amsfonts}
\usepackage{fancyhdr}
\usepackage{comment}

\usepackage{mathrsfs}                           
\usepackage{sidenotes}
\usepackage{pdfpages} 
\newtheorem{thm}{Theorem}[section]
\newtheorem*{thmnonum}{Theorem}
\newtheorem{prop}[thm]{Proposition}
\newtheorem{cor}[thm]{Corollary}
\newtheorem{lemma}[thm]{Lemma}

\newtheorem{definition}{Definition}
\newtheorem*{remark}{Remark}
\theoremstyle{definition}
\newtheorem*{example}{Example}

\newcommand{\R}{\mathbb{R}}

\newcommand{\N}{\mathbb{N}}
\newcommand{\p}{\partial}
\newcommand{\ol}{\overline}
\newcommand{\Tr}[1]{\mbox{Tr}\left(#1\right)}
\newcommand{\Trg}[2]{\mbox{Tr}_{#2}\left(#1\right)}
\newcommand{\Det}[1]{\mbox{Det}\left(#1\right)}
\newcommand{\Detg}[2]{\mbox{Det}_{#2}\left(#1\right)}
\newcommand{\tr}[1]{\mbox{tr}\left(#1\right)}

\renewcommand{\l}{\left}
\renewcommand{\r}{\right}
\newcommand{\la}{\langle}
\newcommand{\ra}{\rangle}\renewcommand{\to}{\rightarrow}

\renewcommand{\exp}[1]{~\mbox{exp}\left(#1\right)}

\renewcommand{\a}{\alpha}
\renewcommand{\b}{\beta}
\renewcommand{\det}[1]{~\mbox{det}\left(#1\right)}
\newcommand{\W}{W^{1,n}_{\textit{loc}}}

\renewcommand{\d}{\delta}
\newcommand{\e}{\epsilon}
\newcommand{\M}{\mathcal{M}}

\newcommand{\abs}[1]{\lvert #1 \rvert}

\date{January 8 2013}
\author{Tony Liimatainen}
\title{On the role of Riemannian metrics in conformal and quasiconformal geometry}
\maketitle
\begin{document}
\includepdf[pages=-6, noautoscale]{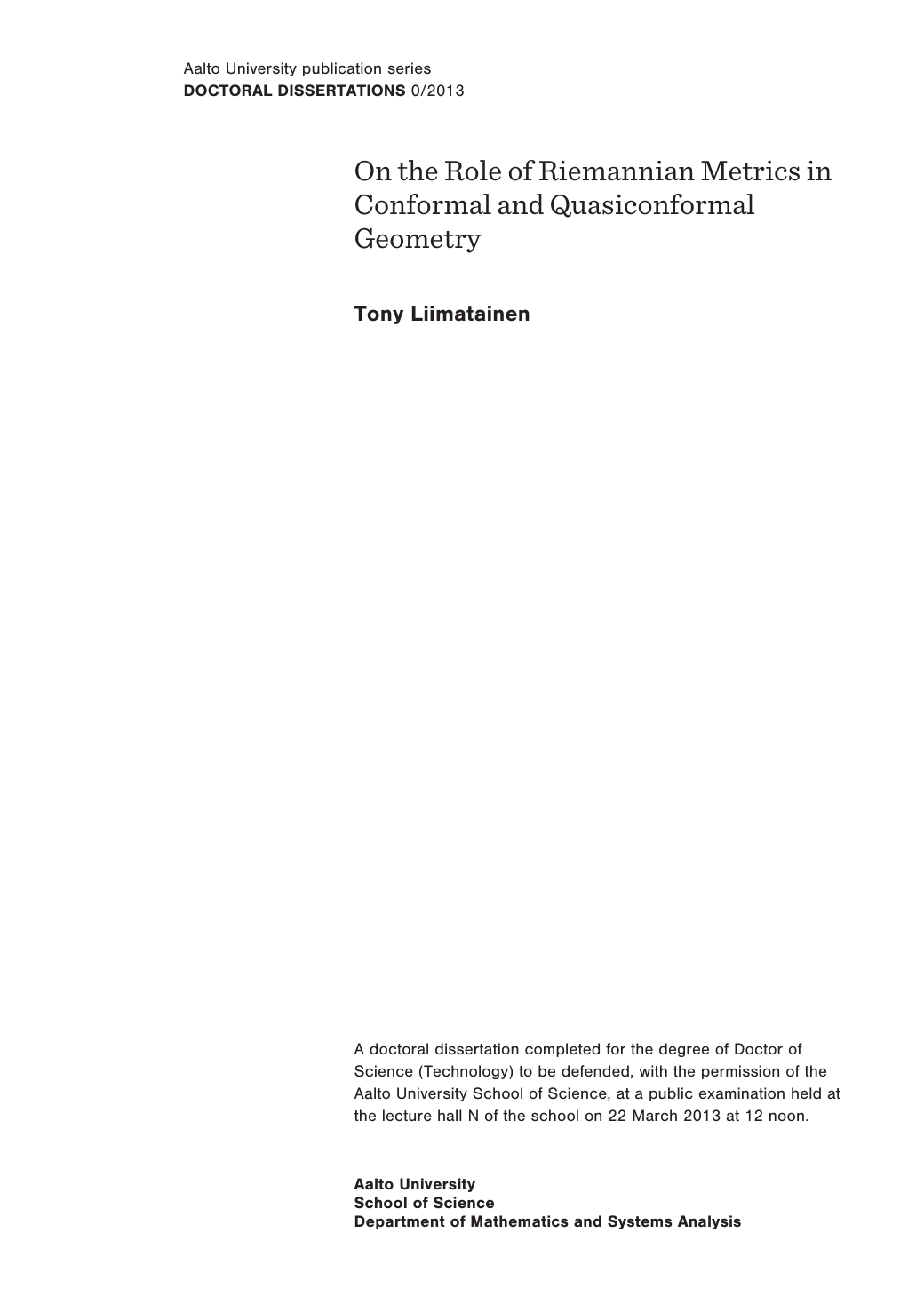}
\draftabstract{

Conformal geometry and the theory of quasiconformal mappings are branches of mathematics that have a broad spectrum of applications ranging from theories in modern physics to recent applications in the inverse conductivity problem.
Notably, the inverse conductivity problem is an especially active area of research with significant contributions from Finnish researchers. The fundamental quantities in these branches of mathematics are Riemannian metrics that many times define, or emerge as solutions to, nonlinear partial differential equations. Partly due to the nonlinearity of the equations in question, these branches of mathematics are still far from being complete theories. 

This thesis focuses on the role of Riemannian metrics in conformal geometry and on the theory of quasiconformal mappings. In general the results of this thesis imply that there is no simple classification of conformal mappings on Riemannian manifolds. A new coordinate invariant definition of quasiconformal and quasiregular mappings is presented and the basic properties of the new class of mappings are established. It is shown that any countable quasiconformal group on a Riemannian manifold (in the introduced sense) can be regarded as a group of conformal mappings with respect to another, optimal, Riemannian metric. In a converse manner, another result of this thesis shows that any smooth manifold of dimension 3 or higher admits infinitely many Riemannian metrics such that there is no conformal diffeomorphisms on the manifold.

The principle of how to find an optimal Riemannian metric for a group of mappings is developed further. It is shown that if the action of a volume form preserving diffeomorphism has a bounded orbit in the space of Riemannian metrics, then a new Riemannian metric can be found such that the diffeomorphism can be regarded as an isometry. The proof of this result relies on generalizations of Neumann's mean ergodic theorem and a fixed point theorem to certain nonpositive curvature metric spaces. The generalizations are formulated and proven in this thesis.

Finally, implications of the regularity of Riemannian metrics in conformal geometry are studied. A new proof of a regularity theorem of conformal mappings between two Riemannian manifolds is achieved. The proof is based on a new coordinate system that generalizes both the harmonic coordinates and the isothermal coordinates. The existence of such coordinates on any Riemannian manifold is established. Additionally, a convergence theorem for conformal mappings is given.
}

\draftabstract[finnish]{

Konformigeometria ja kvasikonformikuvausten teoria ovat matematiikan aloja, joilla on laajalti sovelluksia moderneista fysiikan teorioista viimeaikaisiin sovelluksiin käänteisessä johtavuusongelmassa. Erityisesti käänteisen johtavuuden ongelma on huomattavan aktiivinen tutkimuksen ala, jossa suomalaiset tutkijat ovat saavuttaneet huomattavia tuloksia. Keskeisssä roolissa näissä matematiikan aloissa ovat Riemannin metriikka, jotka usein määrittelevät epälineaarisia osittaisdifferentiaaliyhtälöitä tai esiintyvät niiden ratkaisuina. Johtuen kyseisten yhtälöiden epälineaarisuudesta, nämä matematiikan alat ovet edelleen alati kehittyviä.

Tämä väitöskirja keskittyy tarkastelemaan Riemannian metriikoiden asemaa konformigeometriassa ja kvasikonformikuvausten teoriassa. Yleisellä tasolla voidaan tämän väitöskirjan tuloksiin nojaten sanoa, että Riemannin moniston konformikuvauksille ei ole yksinkertaista luokittelua. Väitöskirjassa esitellään uusi koordinaateista riippumaton määritelmä kvasikonformikuvauksille ja määriteltyjen kuvausten perusominaisuudet johdetaan. Näytetään, että numeroituva kvasikonformikuvausten ryhmä Riemannin monistolla (uuden määritelmän mielessä) voidaan ymmärtää ryhmänä konformikuvauksia toisen, optimaalisen, Riemannin metriikan suhteen. Toinen tämän väitöskirjan tulos osoittaa toisaalta, että millä tahansa vähintään kolmiulotteisella sileällä monistolla on äärettömän monia Riemannin metriikoita, joilden suhteen monistolla ei ole ainuttakaan konformikuvausta.

Väitöskirjassa kehitetään periaatetta, jolla optimaalinen Riemannin metriikka ryhmälle kuvauksia voidaan löytää. Osoitetaan, että jos tilavuusmuodon säilyttävällä diffeomorfismilla on rajoitettu rata Riemannin metriikoiden avaruudessa, niin tällöin löytyy toinen Riemannin metriikka, jonka suhteen kyseinen diffeomorfismi voidaan mieltää isometriana. Tuloksen todistus nojaa Neumannin ergodisen teoreeman ja kiintopistelauseen yleistyksiin tiettyihin ei-positiivisen kaarevuuden metrisiin avaruuksiin. Nämä yleistykset on muotoiltu ja todistettu tässä väitöskirjassa.

Lopuksi tässä työssä tutkitaan, miten Riemannin metriikoiden säännöllisyyden vaikutus ilmenee konformigeometriassa. Uusi todistus konformikuvausten säännöllisyydelle löydetään. Uusi todistus nojautuu uuteen koordinaattisysteemiin, joka yleistää sekä harmoniset että isotermiset koordinaatit. Kyseisten koordinaattien olemassaolo yleisellä Riemannin monistolla todistetaan. Tuloksien avulla osoitetaan suppenemistulos konformikuvauksille.
}


\begin{preface}[Espoo]
The interplay between mathematics and physics has been intriguing me throughout my studies. Before starting my PhD work I had also background in theoretical physics (from the University of Helsinki). I believe that partly because of that I found writing this thesis especially rewarding for that I could apply ideas and concepts from physics to research problems in pure mathematics. During my studies and this PhD work I have been blessed to be surrounded by so many great and curious academic and nonacademic people. Their interest in my studies and work has been most motivational and beneficial for me. This is my thanks to all of you.

The opportunity to write this thesis was given to me by my supervising professor Olavi Nevanlinna and my thesis advisor Kirsi Peltonen. I wish to thank them for all the work they have done to to make this happen. I also wish to thank the Finnish National Graduate School in Mathematics and its Applications that made this thesis financially possible and its director of Hans-Olav Tylli for supporting my various academic travels. 

Two out of the three papers in this thesis are coauthored by professor Mikko Salo who I have had the pleasure to work with. Without his knowledge and contribution this thesis would have turned out very different. Thank you Mikko for everything!

I wish to thank professors Tadeusz Iwaniec and Jeremy Tyson for pre-examining my thesis and their interest in it. Thanks to professors Juha Kinnunen and Antti Kupiainen as well as to Jarmo Malinen for their interest in my ideas in mathematics and giving valuable feedback on them. I would also like to thank professor Sylvie Paycha for inviting me to a conference held in Potsdam in 2011. Sylvie's enthusiasm for mathematics has been an inspiration for me.

During my studies I have had the pleasure to discuss mathematics and physics and beyond with my fellow students Henri Lipponen and Teppo Mattsson who are never short of opinions. Thank you for those numerous occasions. Special thanks I wish to address to Kurt Baarman and Heikki Apiola who I have shared an office during my PhD work.

My warmest thanks I wish to address to my parents, to my love Reija and all my relatives. Your support and faith in my studies and work has been invaluable. Thanks to all my great friends of which I especially wish to thank Jussi Voipio for proofreading a variety of my works.

Finally I wish to thank Aalto university (previously Helsinki University of Technology) in general, which is a great multidisciplinary university filled with inspiring people and activities.
\end{preface}

\tableofcontents

\listofpublications



\chapter{Introduction}
Laws of physics are ever increasingly explained by theories that are based on symmetry principles. For example, our understanding of the Universe in the large relies on the Einstein's theory of gravity, the General relativity. To incorporate the fact that the speed of light is constant with the symmetry principle that the laws of physics should be the same for every observer, Einstein formulated the General relativity as a theory of curved space and time where observable quantities are derived from geometrical objects. 

It is not only Einstein's theory of relativity that is based on symmetry principles nor it is uncommon that various symmetries emerge when statistical properties of physical systems are studied. An interesting fact about symmetries in physics is that many of them, if not most, have a geometrical interpretation. The symmetry principle of Einstein's general relativity is called coordinate invariance, which lies at the heart of modelling geometry mathematically. Another symmetry principle that is relevant to this thesis, and that is also present in some physical phenomena, is the conformal invariance. Roughly speaking, a physical theory satisfies conformal invariance, is conformally invariant, if the theory reacts to deviations of directions, but is oblivious to the magnitude of the deviations. 

The two examples of symmetry principles were mentioned intentionally. The subject of this thesis is the mathematical theory of conformal symmetry, its deviations and the unification of these two with the principle of coordinate invariance. This mathematical theory belongs to a branch of mathematics called conformal geometry, or quasiconformal geometry, if small deviations from conformal symmetry are allowed.

Conformal geometry is a field of mathematics that studies angle preserving transformations, conformal mappings, and how geometrical objects behave under such transformations. Conformal geometry has a long history that can be dated to 18th century and to the studies of complex functions by d'Alembert and Euler~\cite{Bell}. The advances in this field were continued by the works of Cauchy and Riemann in the 19th century. Deriving from the works of d'Alembert and Euler, their works concentrated on a set of equations that are today called the Cauchy-Riemann equations. The Cauchy-Riemann equations define complex differentiable functions, which have the consequential property that these functions preserve the angles between (essentially all) vectors on the complex plane. Complex differentiable functions are conformal transformations.

In the early 20th century, a generalization of conformal mappings were introduced in the works by Gr\"otzsch and by the Finnish mathematician Lars Ahlfors~\cite{Ahlfors, Iwaniec}. These generalizations are called quasiconformal mappings, as coined by Ahlfors. They have the geometrical property that even though they are not conformal mappings in general, these mappings can distort angles only in a controlled manner. Quasiconformal mappings are nearly conformal mappings. Not much later on, still in the early and mid-20th century, quasiconformal mappings were successfully applied in solutions of several problems. Quasiconformal mappings were applied most notably by Teichm\"uller in his studies of extremal mappings between Riemannian surfaces and by Drasin in the inverse Nevanlinna problem (the problem is named after a Finnish mathematician Rolf Nevanlinna)~\cite{Iwaniec, Drasin, Teich}. Riemannian surfaces are two dimensional surfaces in which one can measure angles between intersecting paths on the surface.

The theory of quasiconformal mappings was extended in the late 60's to consider mappings on higher dimensional spaces. The research in this field was lead by Reshetnyak and by the Finnish school of Martio, Rickman and V\"ais\"al\"a~\cite{Rickman, Iwaniec, Reshetnyak, MRV1, MRV2, MRV3}. 

Surfaces in higher dimensional spaces are called manifolds, and a manifold in which angles between intersecting paths can be measured is a Riemannian manifold. The mathematical device that measures angles is called a Riemannian metric. In Einstein's general relativity, a Riemannian metric is the physical quantity that determines the curvature of space and time and it also determines trajectories of particles. Riemannian metrics obey the symmetry principle of coordinate invariance as Einstein desired. 

The concept, the definition, of conformal mappings easily extends for mappings between Riemannian manifolds since their Riemannian metrics allow one to determine if angles are preserved or not. After the studies by Ahlfors and Teichm\"uller and others, the definition of quasiconformal mappings has been extended from mappings on the complex plane, and from mappings between Riemannian surfaces, to apply to mappings between manifolds. These extended definitions of quasiconformality, however, have not fully been compatible with the symmetry principle of coordinate invariance. This inconsistency is addressed in this thesis. The main focus of this thesis lies in the coordinate invariance and its unification with the concept of quasiconformality. The author hopes that this unification, together with the research included in this thesis, will open new applications of the theory of quasiconformal mappings in geometry, but also in physics.

This thesis presents a coordinate invariant method to study quasiconformal mappings between Riemannian manifolds. In the introduced framework coordinate invariance and quasiconformality are studied simultaneously. The coordinate invariant method employs the Riemannian metrics on the manifolds that are used to measure if and how much angles are perturbed by mappings. The new method is applied in this thesis to study the role of Riemannian metrics in questions that are related to conformal and quasiconformal mappings. The original foundations of conformal mappings were in the study of complex functions and thus also in analysis. This thesis continues that tradition by emphasizing the analytical aspects of the developed theory.

The main contributions of this thesis are to the field of conformal geometry and to the field of quasiconformal mappings in mathematics. The theory of quasiconformal mappings, the theory of \emph{Riemannian quasiconformal mappings}, presented in this thesis, is a new approach to the study of quasiconformal mappings and their non-injective counterparts, quasiregular mappings, between Riemannian manifolds. In this summary, the basics of this theory are established and an application of the theory is given. The application demonstrates the usability of the introduced theory. 

The articles that are included in this thesis study questions in conformal geometry and in quasiconformal geometry by the ideas and the tools developed in this summary. The main results of the articles consider the nonexistence of conformal mappings in certain geometries, optimal geometries for certain classes of mappings and the regularity theory of conformal mappings. The articles also contain results that rely only on well established concepts and can therefore be used independently of this summary. In particular, the articles contain new tools that can be used in mathematics beyond the scope of this thesis.

In addition to the works already mentioned, there are of course many other works related to this thesis. Quasiconformal mappings between manifolds are studied and applied in the works of Lelong-Ferrand and Mostow~\cite{FerrandB, Mostow}. In a more general setting, quasiconformal mappings between metric spaces have been studied notably by Heinonen and Koskela in~\cite{HK}. The study of optimal geometries in this thesis relies on the works by Tukia, Iwaniec and Martin~\cite{Tukia, Martin}. The latter of these references considers uniformly quasiregular mappings that can be considered as rational mappings in higher dimensions~\cite{Iwaniec}. The works~\cite{Dairbekov, Reimann1, Reimann2, Holopainen} on quasiconformal and quasiregular mappings on Heisenberg and Carnot groups, equipped with a sub-Riemannian metric, serve as a point of comparison for this summary.

Plausible future applications of the results of this thesis include the following. The inverse conductivity problem of Calder{\'o}n considers the determination of the conductivity of a medium from current and voltage measurements at the boundary of the medium. In the case the medium lies on a plane, quasiconformal mappings have recently shown to be an effective tool in a solution of the inverse conductivity problem of Calder{\'o}n~\cite{AstalaPaivarinta, AstalaPaivarintaLassas}. The tools and concepts presented in this thesis enable one to study to what extent the techniques and ideas in the plane case generalize into more general settings. The inverse problem of Calder{\'o}n on Riemannian manifolds is still an open problem~\cite{DKSaU}. The studies of optimal geometries presented in this thesis could open new interesting questions in the study of perturbations of conformal invariance present in some statistical physics models of percolation and in the study of the related Schramm-Loewner evolution equation~\cite{AnttiK}. More general future applications include studies in conformal geometry and its connection to physics, the study of harmonic mappings between Riemannian manifolds and nonlinear elasticity, and studies concerning infinite dimensional geometries of spaces of Riemannian metrics. On the latter two subjects, the relevant works include~\cite{Jost, Schoen1, IwaniecOnninen, IwaniecKovalevOnninen} and~\cite{Clarke, Freed, Tromba} respectively.

The rest of the summary is divided into three chapters: Chapter 2 develops the theory of Riemannian quasiregular mappings, Chapter 3 gives an application of theory developed in Chapter 2 and Chapter 4 contains a detailed overview of the articles included in this thesis.





\chapter{Quasiregular mappings and distortion on Riemannian manifolds}\label{ch1}
The standard analytic definition of a quasiregular mapping $\phi:\Omega \rightarrow \R^n$, $\Omega$ open in $\R^n$, is given by the \emph{distortion inequality} 
\begin{equation}\label{def}
 ||D\phi||_{op}^n\leq K J_\phi.
\end{equation}
Here $D\phi$ is the differential of $\phi$ and
\begin{equation*}
 ||D\phi||_{op}=\sup_{|X|=1}|(D\phi)X|
\end{equation*}
is its operator norm. The norms of the vectors are given by the Euclidean inner product and $K\geq 1$ is the \emph{quasiconformality constant}. The natural regularity assumption is that the mapping $\phi$ belongs to the Sobolev space $\W(\Omega,\R^n)$ and~\eqref{def} is assumed to hold for the weak differential $D\phi$ of $\phi$ a.e. The (weak) Jacobian determinant $J_\phi=\det{D\phi}$ is assumed to be non-negative or non-positive a.e. on $\Omega$. If $\phi$ is in addition a homeomorphism, then $\phi$ is called quasiconformal~\cite{Iwaniec, Reshetnyak, Rickman, Vaisala}.

The definition generalizes to mappings between manifolds by declaring that the distortion inequality holds in all local coordinates, but in the resulting definition, the quasiconformality constant $K$ depends on the chosen coordinates. The reason for the coordinate dependence is that the Jacobian determinant of a local coordinate representation of a mapping depends not only on the point on the manifold, but also on the chosen coordinates.

Let us explain why the Jacobian determinant of a local coordinate representation of a mapping is coordinate dependent. The Jacobian matrix of a mapping depends on the coordinates of both the domain and target spaces. A coordinate transformation on either space accounts for multiplication of the Jacobian matrix of the mapping by the coordinate transformation matrix, the differential of the transition function. Consequently, the Jacobian determinant of the mapping in the transformed coordinates depends also on the Jacobian determinant of the transition function. 

The source of above the coordinate dependence is that the differential $D\phi$ of a mapping $\phi:M\to N$ at a point $p\in M$ is a linear mapping between the different tangent spaces $T_pM$ and $T_{\phi(p)}N$, while the determinant is defined basis independently only for linear mappings from a vector space to itself. As the tangent spaces of $\R^n$ can be identified, there is no inconsistency within the defining equation~\eqref{def} on $\R^n$. For mappings between manifolds a generalized notion of Jacobian determinant is usually defined by taking the pullback of the volume form on the target manifold~\cite{HIMO}. Using this approach, the Jacobian determinant can be viewed as an $n$-form on $M$ with $n=\dim(M)$, and, if this $n$-form is compared to the volume form of the domain manifold, it yields a function representing the Jacobian determinant. That function is the definition of a Jacobian determinant we use although we motivate the definition somewhat differently.


The \emph{standard definition} of quasiregular (or quasiconformal) mappings between manifolds uses coordinate charts whose transition functions are quasiconformal mappings on $\R^n$~\cite[Ch. 1.6]{Iwaniec}. A collection of such coordinate charts, an atlas, is called a \emph{quasiconformal structure}. In this case, a mapping is said to be quasiregular if~\eqref{def} holds in all the local coordinates of the chosen atlas for some finite $K$. The quasiconformality constant $K$ depends on the chosen atlas and may vary if coordinate charts outside the atlas are used. This is the approach already mentioned above. The definition is not coordinate invariant, but the qualitative condition that $K$ is bounded is sufficient in many applications. This definition of quasiregularity has been successfully applied in several studies; see e.g.~\cite{Mostow, Heinonen, FerrandA}. The definition also extends to topological manifolds which are not even $C^1$ smooth: except in dimension $4$, quasiconformal structures can be found even for topological manifolds~\cite{Iwaniec, Donaldson}. 

The other common definition is the \emph{metric definition} of quasiconformal mappings on metric spaces, which is valid also for Riemannian manifolds where distance functions are given by their Riemannian metrics~\cite{Rickman, Koskela}. We also mention the definition of quasiregular mappings between Riemannian spaces by Reshetnyak~\cite[Ch. I.5.]{Reshetnyak}. The governing principle of his definition is closely related to that in the one we give. 

We present a new definition for quasiregular mappings between Riemannian manifolds that generalizes the definition of quasiregular mappings on $\R^n$. The mappings satisfying the new definition are called \emph{Riemannian quasiregular mappings}. The definition is independent of the choice of local coordinates and uses the standard framework of tensor analysis. The freedom to make any choice of coordinates allows us to apply techniques and formulas from Riemannian geometry. This fact also establishes new connections to other branches in Riemannian geometry and to geometric analysis. See~\cite{Liimatainen1} for a first step in this direction. The freedom to work in arbitrary coordinates might be of practical interest also for the study of quasiregular mappings on $\R^n$.

The new definition of quasiregular mappings we give, provides systematic and quantitative tools to study global problems concerning quasiregular mappings and Riemannian geometry. It is the coordinate dependence of the standard definition that can easily lead to a bookkeeping issues on the use of different coordinate charts. For an example on this matter, see~\cite{Riikka} where a rigorous treatment of a result on the existence of quasiregular mappings from $\R^n$ to a manifold $M$ is presented. By using the new coordinate invariant definition, some of the arguments in that work can be seen to simplify. From the view point of Riemannian geometry another shortage of the standard definition is the ambiguity in the quasiconformality constant $K$. The distortion of a mapping, that on Riemannian manifolds is measured by the shapes of the images of tangent spheres, is only qualitative in the standard definition.

The metric definition is coordinate invariant, but our definition applies in problems that are analytic in nature. The defining distortion inequality of the new Riemannian quasiregular mappings is analogous to~\eqref{def}. We give an application of the theory we develop to a global existence problem that illustrates the advantages of the new definition. This application is explained below in more detail. We also mention that the definition we give extends readily to a definition of mappings between Riemannian manifolds of finite distortion.

The main subjects of the present work are the basic geometric and analytic properties of Riemannian quasiregular mappings. Theorem~\ref{rqr} illustrates the characteristic properties of Riemannian quasiregular mappings and implies that these mappings share the same analytic properties as the quasiregular mappings defined with respect to a quasiconformal structure. Theorem~\ref{uniform_conv} is a natural convergence theorem for Riemannian quasiregular mappings. Some of the results we present are directed to be used in a subsequent paper that considers regularity of conformal mappings on Riemannian manifolds~\cite{LiimatainenSalo}. That paper presents a natural application of the theory we develop and can be considered as a partial motivation for the present work.

We give a geometric application of the new definition. The application generalizes a result by Tukia concerning invariant conformal structures for quasiconformal groups~\cite{Tukia}. We formulate and prove a theorem, which shows that every countable quasiconformal group on a Riemannian manifold admits an invariant conformal structure. Previously the result and its generalization to abelian semigroups have been proven only for $\R^n$, its open subsets and for $n$-spheres~\cite{Tukia, Martin, Iwaniec}. By our definition, Tukia's result generalizes straightforwardly to general Riemannian manifolds. In addition, we discuss how to generalize the theorem we give to abelian semigroups. See also~\cite{Liimatainen1} for a recent related result.  

To prove the existence of an invariant conformal structure for a quasiconformal group, one has to construct a fiber bundle over the manifold whose sections are conformal structures. The construction of this bundle is presented and we show that the bundle admits an elegant geometry. The geometry of the bundle is inherited form the natural geometry of the set $SL(n)/SO(n)$ of positive definite symmetric determinant one matrices.

\section{A Riemannian definition of quasiregular mappings}\label{section_riemannian_def}
The new \emph{Riemannian definition} of quasiregular and quasiconformal mappings is given via the Riemannian metrics on the manifolds. 

\begin{definition}\label{idef}
Let $\phi:(M,g)\to (N,h)$ be a localizable $\W(M,N)$ mapping between Riemannian manifolds with continuous Riemannian metrics. In this case, the mapping $\phi$ is said to be \emph{Riemannian $K$-quasiregular} if the Jacobian determinant of $\phi$ has locally constant sign and if it satisfies the distortion inequality
\begin{equation}\label{ieq}
 \emph{Tr}_g\,(\phi^*h)^n \leq K^2\, \emph{Det}_g\,(\phi^*h) \mbox{ a.e.}
\end{equation}  
If the mapping $\phi$ is in addition a homeomorphisms, then $\phi$ is called \emph{Riemannian $K$-quasiconformal}.
\end{definition}

Clarification of the details are in order. The \emph{invariant normalized trace} $\mbox{Tr}$ and the \emph{invariant determinant} $\mbox{Det}$ above for a general $2$-covariant tensor field $T=(T_{ij})\in T^{2}_{0}(M)$ are given in local coordinates by
\begin{equation*}
\begin{split} 
 \Trg{T}{g}=&\frac{1}{n} \tr{g^{-1}T} \\
 \Detg{T}{g}=&\det{g^{-1}T}.
\end{split}
\end{equation*}
Here the trace and the determinant on the right hand sides are the usual ones for the matrix product $g^{-1}T$ of the representation matrices of $g^{-1}$ and $T$ with respect to any local frame and coframe. Occasionally we omit the subscript $g$ from $\Trg{\cdot}{g}$ and $\Detg{\cdot}{g}$ if the Riemannian metric used in the definition is clear from the context.

The Sobolev space $W^{1,p}(M,N)$ of mappings between Riemannian manifolds is usually defined by isometrically embedding the target manifold $N$ to some $\R^k$, where $k$ is large enough, and then defining
\begin{equation}\label{sobo}
 W^{1,p}(M,N)=\{u\in W^{1,p}(M,\R^k): u(x)\in N \mbox{ a.e.}\}.
\end{equation}
Here $u$ is understood as the composition $E\circ\phi$, where $\phi:M\to N$ and $E$ is an embedding of $N$ to some $\R^k$, implicit in the definition. See~\cite{HIMO, Hajlasz2, Schoen1, Schoen2} for details about this definition. For our purposes the following simpler definition of Sobolev spaces is enough. This is because we assume that the mappings we consider are localizable. 

A mapping $\phi$ is \emph{localizable} if for every $p\in M$ there exist a neighborhood $U$ of $p$ and a coordinate neighborhood $V$ of $\phi(p)$ such that $\phi(U)$ is compactly contained in $V$, $\phi(U)\subset\subset V$. If $p\in M$, and $U$ and $V$ are as described, the restriction $\phi:U\to V$ is called a \emph{localization} of $\phi$ at $p$ or simply a localization. Note that every continuous mapping is localizable. In fact, Riemannian quasiregular mapping can be redefined on sets of measure zero to continuous mappings (see Thm.~\ref{rqr}), and therefore we could have assumed continuity in the definition of Riemannian quasiregular mappings above without any loss of generality. We assume only localizability of the mappings, since the emphasis of the present work is in its functional analytic aspects, where continuity is implied by the theory.

The localizability of the mappings allows us to define Sobolev spaces in terms of local coordinate charts. We say that $\phi$ belongs to the Sobolev space $W^{1,p}(M,N)$, $p>1$, if the coordinate representation of every localization $\phi:U\to V$ is in the Sobolev space $W^{1,p}(U,V)$ in $\R^n$. Here $U$ and $V$ are identified with the images of $U$ and $V$ under coordinate charts. This is the definition of Sobolev spaces we use. We remark that the same definition of Sobolev spaces has also been used in the study of harmonic mappings between manifolds~\cite{Jost}. Local Sobolev spaces $W^{1,p}_{loc}(M,N)$ are defined in an analogous manner.

Assume that a mapping is localizable and that it belongs to the Sobolev space $W^{1,p}(M,N)$ in the sense of~\eqref{sobo}. Then the mapping has the property that local coordinate representations of the mapping belong to $W^{1,p}$ in $\R^n$. This follows from the fact that if $U$ and $(V,\{y^i\})$ are as in the definition of localizability, we can write
$$
\phi^i=y^i\circ\phi=(y^{i}\circ E^{-1})\circ (E \circ \phi),
$$
where $E:M\to N$ as in~\eqref{sobo}. Therefore, $\phi^i$ is a function on an open subset $U$ of $\R^n$ that belongs to the Sobolev space $W^{1,p}(U,\R)$ as a composition of a smooth and a Sobolev mapping; see e.g.~\cite[Thm. 6.16]{Lieb}. It follows that mappings belonging to the Sobolev space $W^{1,p}(M,N)$ in the sense of~\eqref{sobo} satisfy the definition of Sobolev mappings we use.

We define the \emph{weak differential} $D\phi$ of $\phi$ in terms of the weak differentials of the component functions of the coordinate representation of the localizations of $\phi$. The chain rule of differentiation, valid for compositions of Sobolev and $C^1$ smooth mappings~\cite[Ch. I.2.5.]{Reshetnyak}, applied to the transition functions, shows that $D\phi$ is a well-defined bundle map $TM\to TN$ mapping each tangent space $T_pM$ to $T_{\phi(p)}N$. The proof is the same as in the smooth case. Naturally, the weak differential is defined only modulo sets of measure zero. The (weak) pullback of the Riemannian metric $h$ that appears in the distortion inequality is defined with respect to the weak differential by the usual formula
$$
\phi^*h=D\phi^Th|_\phi D\phi.
$$
The Jacobian determinant of $\phi$ has \emph{locally constant sign} if the Jacobian determinant of every coordinate representation of every localization is either non-negative or non-positive a.e.

We assume that the manifolds $M$ and $N$ are of equal dimension $n\geq 2$, oriented and $C^1$ smooth. The Riemannian metrics $g$ and $h$ are continuous unless otherwise stated. The measures $\mu_g$ and $\mu_h$ denote the measures defined by the Riemannian volume forms of $g$ and $h$. These measures are equivalent (sets of measure zero are the same) to the ones constructed by pulling back the Lebesgue measures on $\R^n$ by using local coordinates and partitions of unity. The equivalence is due to the continuity of $g$ and $h$.

\subsection{Motivation}

Let us motivate the given definition of Riemannian quasiregular mappings. The initial motivation for the definition comes from the equation
\begin{equation*}
 \phi^*h=c\,g
\end{equation*}
defining conformal mappings on Riemannian manifolds. The conformal factor $c$ can be solved from the equation in several ways. By taking the invariant normalized trace and the invariant determinant, one has
\begin{equation*}\label{solve_c}
 c=\Trg{\phi^*h}{g} \mbox{ and } c^n=\Detg{\phi^*h}{g}.
\end{equation*}
Equating the formulas for $c$ yields
\begin{equation*}
 \Trg{\phi^*h}{g}^n = \Detg{\phi^*h}{g}.
\end{equation*}
A relaxation of this equality to an inequality with a factor $K^2\geq 1$ gives the distortion inequality~\eqref{ieq}. A conformal mapping is $1$-quasiregular, and the converse is also true as is shown in Proposition~\ref{1qc} (at least if the regularity properties are disregarded). The appearance of the square of $K$ will become apparent later. 

There is a geometrical way to interpret the definition of Riemannian quasiregular mappings. The tangent spaces of Riemannian manifolds $M$ and $N$ are equipped with inner products $g$ and $h$. Thus the (weak) differential of $\phi$ at $p\in M$ is a linear mapping between inner product spaces,
\begin{equation*}
 D\phi:(T_pM,g_p)\rightarrow (T_{\phi(p)}N,h_{\phi(p)}).
\end{equation*}
It follows that there exists the (formal) adjoint 
\begin{equation*}
D\phi^*: (T_{\phi(p)}N,h_{\phi(p)})\rightarrow (T_pM,g_p)
\end{equation*}
of $D\phi$ satisfying
\begin{equation*}
 g(D\phi^* U,V)=h_\phi(U,D\phi V),
\end{equation*}
for all $U\in T_{\phi(p)}N$ and $V\in T_pM$ at each point $p\in M$. Let us calculate the explicit form of the adjoint.

Let $p\in M$ and choose some local frames on neighborhoods of $p\in M$ and $\phi(p)\in N$. Let us denote by $\la \cdot,\cdot\ra$ the standard Euclidean inner product of vectors. From the definition of the adjoint we have
\begin{equation*}
\begin{split}
g(D\phi^*U,V)&=h_\phi(U,D\phi V)=\langle U,h_\phi D\phi V\rangle =\langle (h_\phi D\phi)^TU, V\rangle\\
  &=\langle D\phi^T(h_\phi)^T U,V\rangle=\langle D\phi^Th_\phi U,g^{-1}g V\rangle \\
  &=\langle (g^{-1})^T D\phi^Th_\phi U,g V\rangle=g(g^{-1} D\phi^Th_\phi U,V).
\end{split}
\end{equation*}
Thus the adjoint of $D\phi$ is
\begin{equation*}
 D\phi^*=g^{-1} D\phi^Th_\phi.
\end{equation*}

The normalized Hilbert-Schmidt inner product of linear mappings $T$ and $S$ between inner product spaces is given by
\begin{equation*}
 \langle T,S\rangle=\frac{1}{n} \tr{T^*S}.
\end{equation*}
The induced norm of the (normalized) Hilbert-Schmidt inner product applied to $D\phi$ now yields
\begin{equation*}
 ||D\phi||^2:=\langle D\phi,D\phi\rangle=\frac{1}{n}\tr{g^{-1} D\phi^Th_\phi D\phi}=\Trg{\phi^*h}{g}.
\end{equation*}
Thus, the invariant normalized trace is the square of the normalized Hilbert-Schmidt norm of $D\phi:(TM,g)\rightarrow (TN,h)$. We also define yet another determinant as the square root of $\Det{\phi^*h}$,
\begin{equation*}
 \Det{D\phi}:=\sqrt{\det{D\phi^*D\phi}}.
\end{equation*}

Using what we have just observed, we give an equivalent definition of quasiregular mappings. 

\begin{definition}(Equivalent definition)
 A localizable mapping (homeomorphism) $\phi:(M,g)\rightarrow (N,h)$ of Sobolev class $\W(M,N)$ is $K$-quasiregular ($K$-quasiconformal) if the Jacobian determinant of $\phi$ has locally constant sign and if 
\begin{equation*}
 ||D\phi||^n\leq K\, \emph{Det}\,(D\phi).
\end{equation*}
Here $||D\phi||$ and $\Det{D\phi}$ are the normalized Hilbert-Schmidt norm and the determinant of $D\phi$ defined as
\begin{equation*}
\begin{split}
  ||D\phi||&=\frac{1}{\sqrt{n}}\sqrt{\emph{tr}(D\phi^*D\phi)} \\
  \emph{Det}\,(D\phi)&=\sqrt{\emph{det}(D\phi^*D\phi)}. \\
\end{split}
\end{equation*}
\end{definition}

The definition is just the (positive) square root of Definition~\ref{idef} of quasiregular mappings. It also explains the appearance of $K$ squared in the first definition. 

We have defined traces and determinants that we have claimed to be coordinate invariant. To see the invariance, note first that the mapping $D\phi^*D\phi$ is a mapping from $T_pM$ to itself at each point $p\in M$. Also, a general $\binom{1}{1}$-tensor field can be regarded as a fiber preserving linear mapping $TM\rightarrow TM$. The trace and the determinant of linear mappings from a vector space to itself are independent of the choice of a basis. It follows that the defined traces and determinants of $D\phi$ and $g^{-1}\phi^*h$ and more generally for $\binom{2}{0}$-tensors fields are independent of the local framing. In particular, the definitions are independent of the choice of a local coordinate frame. 

\begin{remark}
A $\binom{1}{1}$-tensor $g^{-1}T$, $T\in T^2_0(M)$, thought as a linear bundle map $TM\rightarrow TM$, has also a well defined eigenvalue equation
\begin{equation*}
 g^{-1}T(X)=\lambda X.
\end{equation*}
The coordinate independent eigenvalues that naturally depend on the point of the manifold are functions on $M$. The eigenvalues are solutions to the coordinate invariant equation
\begin{equation}\label{eig}
 \emph{Det}(T-\lambda g)=\emph{det}(g^{-1}T-\lambda I)=0.
\end{equation}
We will use this fact to make calculations in arbitrary coordinates and local frames.
\end{remark}

We end this section with an example that consider how the flow of a vector field generates distortion.
\begin{example}\label{conformal_flow}
Let $X$ be a (say $C^1$) smooth vector field on $(M,g)$. We study the quasiconformality of the flow $\phi_t$ generated by $X$ by calculating the time derivative of the distortion function $K(t)$.

Define $g(t)=\phi_t^*g$. We take the traces and the determinants with respect to $g$ unless otherwise indicated. Elementary calculations together with the identity
\begin{equation*}
 \frac{d}{dt}\Det{g(t)}=\Det{g(t)}\Trg{\dot{g}(t)}{t}
\end{equation*}
yield
\begin{equation*}
 \frac{d}{dt}K(t)^2=\frac{d}{dt}\frac{\Tr{g(t)}^n}{\Det{g(t)}}=nK(t)^2\l(\frac{\Tr{\dot{g}(t)}}{\Tr{g(t)}}-\Trg{\dot{g}(t)}{t}\r).
\end{equation*}
We can write the equation above as
\begin{equation}\label{timeder}
 \dot{K}(t)=K(t)\frac{n/2}{\Tr{g(t)}} \mbox{Tr}\big(\dot{g}(t)-\Trg{\dot{g}(t)}{t}g(t)\big).
\end{equation}

Let $h$ be a Riemannian metric on $M$, let $T\in T^{2}_{0}(M)$ and let $||\cdot||_h$ denote the usual norm in the bundle $T\in T^{2}_{0}(M)$ induced by the Riemannian metric $h$, $||T||_h^2 = \tr{h^{-1}T^Th^{-1}T}$. The invariant trace of a $\binom{2}{0}$-tensor $T$ satisfies an inequality
\begin{equation*}
 \Trg{T}{g}\leq \Trg{h}{g}||T||_h.
\end{equation*}
To see this, use the Cauchy-Schwartz inequality for the Hilbert-Schmidt inner product of the representation matrices of $g^{-1}$ and $T$ in an $h$-or\-thonor\-mal frame,
\begin{equation*}
 \begin{split}
  n \Trg{T}{g}&=\tr{g^{-1}T} \leq ||g^{-1}||_{HS}||T||_{HS}=\tr{g^{-2}}^{1/2} ||T||_{HS} \\
  &\leq \tr{g^{-1}} ||T||_{HS} = n \Trg{h}{g} ||T||_h.
 \end{split}
\end{equation*}
Applying this inequality with $h=g(t)$ to~\eqref{timeder} gives
\begin{equation*}
 \dot{K}(t)\leq K(t)\frac{n}{2} ||\dot{g}(t)-\Trg{\dot{g}(t)}{t}g(t)||_{g(t)}.
\end{equation*}

The Riemannian inner product, its induced norm and the invariant trace of any tensor $T$ behave naturally under pullbacks,
\begin{equation*}
 ||T||_g|_\phi=||\phi^*T||_{\phi^*g}, \ \ \Trg{T}{g}|_\phi=\Trg{\phi^*T}{\phi^*g}.
\end{equation*}
Also, the time derivative of $g(t)$ is the pullback by $\phi_t$ of the Lie derivative of $g$ in the direction of $X$ (cf.~\cite[p. 13]{Ricci_book}),
\begin{equation*}
 \dot{g}(t)=\phi_t^*\mathcal{L}_Xg=\phi_t^*\dot{g}(0).
\end{equation*}
With these identities we can write
\begin{equation*}
\begin{split}
 ||\dot{g}(t)-\Trg{\dot{g}(t)}{t}g(t)||_{g(t)}&=||\phi_t^*\mathcal{L}_X g-\Trg{\phi_t^*\mathcal{L}_X g}{\phi_t^*g}\phi_t^*g||_{\phi_t^*g} \\
 &=||\mathcal{L}_Xg-\Tr{\mathcal{L}_Xg}g||_g|_{\phi_t}=||SX||_g|_{\phi_t}.
\end{split}
\end{equation*}
The notation $SX$ refers to the Ahlfors operator $S$ applied to $X$. The kernel of $S$ defines conformal Killing vector fields~\cite{Pierzchalski}. Together with Gr\"onwall's inequality we see that the quasiconformality constant under the flow of the vector field $X$ is controlled by the supremum $g$-norm of $SX$:
\begin{equation*}
 K(t)\leq \exp{tn||SX||_\infty/2}.
\end{equation*}
The flow of $X$ is a family of $K(t)$-quasiconformal mappings with $K(t)$ given above. We have recovered an analogous result to the one given in~\cite{Pierzchalski}.
\end{example}

\section{Basic properties of quasiregular mappings}
The definition of Riemannian quasiregular mappings is natural in the Riemannian setting and generalizes the standard definition of quasiregular mappings on $\R^n$. To compare the new definition with the standard definition of quasiregular mappings on $\R^n$, let the manifolds in the defining equation~\eqref{ieq} be Euclidean domains in $\R^n$ with Cartesian coordinates. In Cartesian coordinates the Euclidean metrics $g$ and $h$ are the identity matrices $I$. We have
\begin{equation*}
 \frac{1}{n}||D\phi||_1^2\leq \frac{1}{n}||D\phi||_{HS}^2 = \frac{1}{n}\tr{D\phi^*D\phi}= \Tr{\phi^*h}=\frac{1}{n}||D\phi||_{HS}^2 \leq ||D\phi||_1^2,
\end{equation*}
where we have denoted the operator norm and the (unnormalized) Hil\-bert-Schmidt norm by $||\cdot||_1$ and $||\cdot||_{HS}$ respectively. Moreover, the right hand side of the distortion inequality~\eqref{ieq}, for $g=h=I$,  now reads
\begin{equation*}
 K^2\, \Det{\phi^*h}=K^2\det{D\phi}^2.
\end{equation*}
Together these show that a $K$-quasiregular mapping on $\R^n$ in the standard sense (as defined by~\eqref{def}) is a Riemannian quasiregular mapping with the same quasiregularity factor $K$. 

On the contrary, a Riemannian $K$-quasiregular mapping between Euclidean domains on $\R^n$ is a $n^{n/2}K$-quasiregular mapping in the standard sense. The quasiregularity constant $n^{n/2}K$ in this case is not the best one in general. For example, $1$-quasiregular mappings on $\R^n$ are $1$-quasiregu\-lar in both new and old definitions. This can be seen from the next proposition below. The difference in the quasiregularity constants on $\R^n$ is an effect of the different norm used for the differential of the mapping.

We call the $\binom{1}{1}$-tensor $g^{-1}\phi^*h$ the \emph{distortion tensor}, which is the (unnormalized) generalization of that used in the theory of quasiregular mappings on $\R^n$~\cite[p.100]{Iwaniec}. It is worth noting that the distortion tensor has non-negative coordinate invariant eigenvalues for any (weakly differentiable) mapping $\phi: (M,g)\rightarrow (N,h)$. The coordinate invariance of the eigenvalues follows from the remark of the previous section. The eigenvalues are thus functions on $M$, and, if the mapping is assumed to be of Sobolev class $\W(M,N)$, the eigenvalues belong to $L^{n/2}_{loc}(M)$ due to the inequality
$$
\lambda\leq \tr{g^{-1}\phi^*h}\in L^{n/2}_{loc}(M).
$$
When the mapping $\phi$ is $K$-quasiregular, the ratio of the maximum and minimum eigenvalue is in addition bounded a.e. by $n^nK^{2}$:
\begin{equation}\label{ratio_bound}
 K^2\geq \Tr{\phi^*h}^n/\Det{\phi^*h}=\frac{1}{n^n}\left(\sum_i\lambda_i\right)^n/\prod_i \lambda_i \geq \frac{1}{n^n}\frac{\lambda_{\mbox{max}}}{\lambda_{\mbox{min}}}.
\end{equation}
If $\phi$ is conformal, the only eigenvalue of the distortion tensor is the conformal factor $c$ of the mapping with multiplicity $n$.

\begin{prop}\label{1qc}
A non-constant Riemannian $1$-quasiregular mapping $\phi:(M,g)\rightarrow (N,h)$ is conformal,
\begin{equation*}
 \phi^*h=cg \mbox{ a.e.,}
\end{equation*}
for some positive function $c$. Here $\phi^*h$ is defined by the weak differential of $\phi$.
\end{prop}
\begin{proof}
It suffices to prove the claim locally. Any local matrix representation of $g^{-1}\phi^*h$ has positive eigenvalues $\lambda_i$ a.e. This follows from the fact that $J_\phi\neq 0$ a.e. by Theorem~\ref{rqr} and therefore $g^{-1}\phi^*h$ is a product of symmetric positive definite matrices a.e. 

The inequality of arithmetic and geometric means shows that the function
\begin{equation}\label{balance}
 \Tr{\phi^*h}^n-\Det{\phi^*h}=\left(\frac{1}{n}\right)^n(\lambda_1+\cdots+\lambda_n)^n-\lambda_1\cdots\lambda_n
\end{equation}
has a minimum $0$ for positive $\lambda_i$. This happens exactly when all the $\lambda_i$'s coincide, $\lambda_i=\lambda$. Now, the condition of $1$-quasiregularity yields
\begin{equation*}
 \Tr{\phi^*h}^n\leq\Det{\phi^*h}=\lambda_1\cdots\lambda_n\leq\left(\frac{1}{n}\right)^n(\lambda_1+\cdots+\lambda_n)^n=\Tr{\phi^*h}^n.
\end{equation*}
Accordingly, the minimum of the function in~\eqref{balance} is achieved, and we have
\begin{equation*}
 g^{-1}\phi^*h=A^T\Lambda A=\lambda~\mbox{Id} \mbox{ a.e.}
\end{equation*}
Here $A$ is the diagonalizing orthogonal matrix field of the local coordinate representation of $g^{-1}\phi^*h$. The matrix $\Lambda$ is the diagonal matrix, which has the eigenvalue $\lambda$ (of multiplicity $n$) as the diagonal entries.
\end{proof}

The definition of Riemannian quasiregular mappings can also be interpreted as a definition of quasiconformal metrics by inserting $\phi=\mbox{Id}$ into the definition. We say that two Riemannian metrics $g$ and $h$ are \emph{quasiconformally related} if the relation
\begin{equation*}
 \Trg{h}{g}^n\leq K^2\,\Detg{h}{g}
\end{equation*}
holds for some constant $K \geq 1$. 

The quasiconformality relation, which we denote by $\sim$, is an equivalence relation in the set of Riemannian metrics. This is shown in the proposition below. We write $\mathcal{M}$ for the set of all Riemannian metrics on $M$ with no prescribed regularity properties and $\mbox{Map}(M,\R)$ denotes the general space of functions on $M$.

\begin{definition}\label{dist_func}
The \emph{distortion function} $K^2:\mathcal{M}\times\mathcal{M}\to \emph{Map}(M,\R)$ is the function defined by,
\begin{equation*} 
 K^2(g,h)(p)=\frac{\emph{Tr}_g\,(h)^n}{\emph{Det}_g(h)}(p).
\end{equation*}
\end{definition}

We record for future reference that for any mapping $\phi:M\to N$ the distortion function satisfies 
\begin{equation}\label{functoriality}
 K^2(\phi^*h,\phi^*k)=K^2(h,k)|_\phi
\end{equation}
at the points where the (weak) Jacobian matrix of $\phi$ is invertible. Here $h$ and $k$ are arbitrary Riemannian metrics on $N$. This fact is essentially a statement of the coordinate invariance of $K^2$. 

Note that a mapping satisfies the distortion inequality~\eqref{ieq} precisely when the distortion function applied to $g$ and $\phi^*h$ is bounded on $M$. The distortion function is invariant under conformal scaling of either or both the metrics $g$ and $h$ implying that the notion of quasiregularity is naturally conformally invariant. 

For the distortion function we have the following.
\begin{prop}\label{distortion_function}
The distortion function $K^2:\mathcal{M}\times\mathcal{M}\to \emph{Map}(M,\R)$ of Riemannian metrics on $M$ has the following properties:
\begin{enumerate}
  \item For $g,h, k\in \mathcal{M}$ the distortion function satisfies
\begin{equation}\label{submult}
 K^2(g,k)\leq n^n K^2(g,h)K^2(h,k)
\end{equation}
and
\begin{equation}\label{reflex}
 K^2(g,h)\leq K^2(h,g)^{n-1}.
\end{equation}
\item The quasiconformality relation $\sim$ is an equivalence relation on $\mathcal{M}$.
\end{enumerate}  
\end{prop}
\begin{proof}
Let $g,k$ and $h$ be Riemannian metrics on $M$. We have
\begin{align*}
 K^2(g,k)^{1/n}&= \frac{1}{n} \frac{\tr{g^{-1}k}}{\det{g^{-1}k}^{1/n}}= \frac{1}{n}\frac{\tr{g^{-1}hh^{-1}k}}{\det{g^{-1}hh^{-1}k}^{1/n}} \\
  & \leq\frac{1}{n}\frac{\tr{g^{-1}h}\tr{h^{-1}k}}{\det{g^{-1}h}^{1/n}\det{h^{-1}k}^{1/n}}=n K^2(g,h)^{1/n} K^2(h,k)^{1/n}.
\end{align*}
Here we have used the fact that trace is submultiplicative for a product of positive definite matrices. We have~\eqref{submult}.

To prove~\eqref{reflex}, we use some standard matrix results. The adjugate matrix $\mbox{Adj}(A)$ of an invertible matrix $A$ satisfies
\begin{equation*}
 A^{-1}=\frac{\mbox{Adj}(A)}{\det{A}}.
\end{equation*}
The fact that a symmetric positive definite matrix $A$ has a symmetric square root, applied to a norm estimate 
$$
||\mbox{Adj}(A)||_{HS}\leq n^{-(n-2)/2}||A||_{HS}^{n-1}
$$
of the adjugate matrix (see e.g.~\cite{Mirsky}), yields
\begin{align*}
 \tr{\mbox{Adj}(A)}&=\tr{(\mbox{Adj}(A))^{T/2}(\mbox{Adj}(A))^{1/2}}=\tr{(\mbox{Adj}(A^{1/2}))^T\mbox{Adj}(A^{1/2})} \\
  &=||\mbox{Adj}(A^{1/2})||^2_{HS} \leq n^{-(n-2)} ||A^{1/2}||_{HS}^{2(n-1)} \\
  &= n^{2-n} \tr{A^{T/2}A^{1/2}}^{n-1} = n^{2-n} \tr{A}^{n-1}.
\end{align*}
Here $||\cdot||_{HS}$ denotes the Hilbert-Schmidt norm. Applying the above formulas for $A=h^{-1}g$ yields
\begin{align*}
 K^2(g&,h)=\frac{1}{n^n}\frac{\tr{g^{-1}h}^n}{\det{g^{-1}h}}=\frac{1}{n^n}\frac{\tr{A^{-1}}^n}{\det{A^{-1}}}=\frac{1}{n^n}\frac{1}{\det{A}^n}\frac{\tr{\mbox{Adj}(A)}^n}{\det{A^{-1}}} \\
  &\leq {n^{-n}} n^{n(2-n)}\frac{(\tr{A}^{n-1})^n}{\det{A}^{n-1}} = n^{-n+n(2-n)+ n(n-1)} \bigg(\frac{(\frac{1}{n}\tr{A})^n}{\det{A}}\bigg)^{n-1} \\
  &\leq K^2(h,g)^{n-1}.
\end{align*}
The fact that the quasiconformality relation $\sim$ is an equivalence relation follows.
\end{proof}

We now return to the study of Riemannian quasiregular mappings. Previously we found that Riemannian quasiregular mappings between Euclidean spaces are quasiregular mappings on $\R^n$ and vice versa. On $\R^n$ the new and the standard definition of quasiregularity are equivalent. 

Next we consider a local version of this equivalence for general Riemannian manifolds. We show that if $\phi:(M,g)\rightarrow (N,h)$ is Riemannian quasiregular, then for any $p\in M$ there is a localization of $\phi$ at $p$ such that the coordinate representation (of the localization) is quasiregular mapping on $\R^n$. The quasiregularity constant of the coordinate representation will depend on $K$ and on the coordinate representations of the Riemannian metrics $g$ and $h$.

\begin{thm}\label{localization}
 Let $\phi$ be a Riemannian $K$-quasiregular mapping between Riemannian manifolds $(M,g)$ and $(N,h)$. Then, for any $p\in M$, there is a localization of $\phi$ at $p$ such that the coordinate representation of the localization is a quasiregular mapping on $\R^n$. The quasiregularity constant $\ol{K}$ of the localization depends only on $K$ and on the coordinate representations of $g$ on $U$ and $h$ on $V$
\end{thm}
\begin{proof}
Let $p\in M$. Choose a neighborhood $U$ and a coordinate neighborhood $(V,\varphi)$ around $p$ and $\phi(p)$ such that $\phi(U)\subset\subset V$. We can assume that $U$ is compactly contained in some coordinate neighborhood of $p$ whose coordinate mapping we denote by $\psi$. The coordinate representation of $\phi$ is of Sobolev class $\W(U,V)$. By the continuity of $g$ and $h$ and by compactness we have
\begin{equation*}
  K(\psi^*I,g)\leq K_1\mbox{ on $U$ and } K(h,\varphi^*I)\leq K_2 \mbox{ on } V. \\
\end{equation*}
It is sufficient to show that the distortion of the coordinate representation $\varphi\circ\phi\circ\psi^{-1}:U\to V$ of $\phi$ with respect to Euclidean metrics is bounded. By~\eqref{functoriality} and~\eqref{submult} we have
\begin{align*}
  K^2&(I,(\varphi\circ\phi\circ\psi^{-1})^*I) \leq n^{2n}K^2(I,(\psi^{-1})^*g)\, \\
 &\times K^2((\psi^{-1})^*g,(\phi\circ \psi^{-1})^*h)\,K^2((\phi\circ \psi^{-1})^*h,(\varphi\circ\phi\circ\psi^{-1})^*I) \\
  &=n^{2n}K^2(\psi^*I,g)|_{\psi^{-1}} K^2(g,\phi^*h)K^2(h,\varphi^*I)|_{\phi\circ \psi^{-1}}\leq n^{2n}K_1^2 K^2 K_2^2.
\end{align*}
Here we have used~\eqref{functoriality} and the chain rule of differentiation, valid for compositions of $C^1$ smooth and Sobolev mappings (see e.g.~\cite{Lieb}), to deduce that
$$
(\varphi\circ\phi\circ\psi^{-1})^*I=(\phi\circ\psi^{-1})^*\varphi^*I.
$$
\end{proof}

We are now ready to state the basic properties of Riemannian quasiregular mappings. The following theorem considers mostly the analytic properties of quasiregular mappings, which all follow from the well-known corresponding statements of quasiregular mappings on $\R^n$. The theorem in particularly states that the standard formulas for derivatives and integration of smooth mappings are valid a.e. for Riemannian quasiregular and Riemannian quasiconformal mappings. 

\begin{thm}\label{rqr}
 Let a mapping $\phi:(M,g)\to (N,h)$ be Riemannian $K$-quasiregular and assume that $g$ and $h$ are continuous. Then the following hold:
\begin{enumerate}
\item The mapping $\phi$ is differentiable a.e., and at the points where the differential exists, it coincides with the weak differential. The mapping $\phi$ can be redefined on a set of measure zero to be continuous.
\item Let $u\in \W(N)$. Then $\phi^*u=u\circ \phi$ is of Sobolev class $\W(M)$ and $\phi^*u$ satisfies a.e. the chain rule of derivation:
\begin{equation}\label{chain_rule}
 \p_i(u\circ\phi)=\p_a u|_\phi \p_i\phi^a \mbox{ i.e. }  d\phi^*u=\phi^*du.
\end{equation}
Moreover,
\begin{equation*}
|d(u\circ\phi)|^n_g\leq n^{n} K\, \emph{Det}\,(D\phi)\phi^*(|du|_h^n) \mbox{ a.e.}
\end{equation*}
Here $|\cdot|_g$ and $|\cdot|_h$ are the norms induced by $g$ and $h$ on $1$-forms.
\item If the mapping $\phi$ is non-constant (and redefined to be continuous), then it is open and discrete. The Riemannian Jacobian determinant
\begin{equation*}
 \emph{Det}\,(D\phi)=\sqrt{\emph{det}\,(g^{-1}\phi^*h)}\in L^1_{loc}(M)
\end{equation*}
of $\phi$ is non-vanishing a.e. in this case.
\item If $\psi:(N,h)\to (L,k)$ is another Riemannian quasiregular mapping, with quasiregularity factor $K'$, then the composition $\psi\circ\phi$ is $n^{n/2}K'K$-quasiregular. Moreover, the chain rule holds for the differential of the composition:
\begin{equation}\label{map_combo}
 D(\psi\circ\phi)=D\psi|_\phi\circ D\psi \mbox{ a.e.}
\end{equation}
In particular, we have
\begin{equation}\label{pullback}
  (\psi\circ\phi)_*U=\psi_*\phi_*U \mbox{ and }  (\psi\circ\phi)^*T=\phi^*\psi^*T
\end{equation}
holding a.e. for any vector (field) $U$ of $M$ and covariant tensor (field) $T$ of $L$.
\item If $\phi$ is in addition a homeomorphism, and thus $K$-quasiconformal, then its inverse is $K^{n-1}$-quasiconformal. Also, integration by substitution is valid:
\begin{equation*}
 \int_M f\circ \phi\, \emph{Det}\,(D\phi)\,d\mu_g=\int_{N} f d\mu_h.
\end{equation*}
Here $f$ is any integrable function on $N$.
\end{enumerate}
\end{thm}
\begin{proof}
Since $\phi$ is localizable, we can choose atlases $\{U_\alpha\}_{\alpha\in \mathbb{N}}$ and $\{V_\beta\}_{\beta\in \mathbb{N}}$ of $M$ and $N$ such that for any $\alpha\in \mathbb{N}$ there is $\beta\in \mathbb{N}$ such that $\phi(U_\alpha)\subset\subset V_\beta$. For any $\alpha$ and a corresponding $\beta$, the coordinate representation of the restriction $\phi_{\a\b}:U_\alpha\to V_\beta$ is a quasiregular mapping on $\R^n$ by the previous theorem.

$(1)-(3)$: With respect to the Lebesgue measure on $\R^n$, the coordinate representation $\phi_{\a\b}$ is differentiable a.e., and, at the points where the differential exists, the differential coincides with the weak differential~\cite[p. 84]{Reshetnyak}. 

Any coordinate representation of any localization of $\phi$ can be redefined on a set of measure zero such that the resulting localization is continuous~\cite[Ch. 7]{Iwaniec}. Via the $C^1$ coordinate charts, the same holds for the actual localizations. The redefinitions made in different local coordinates agree: Let $\phi_1$ and $\phi_2$ be continuous redefinitions of $\phi$ on $U_1\subset M$ and on $U_2\subset M$ with $U_1\cap U_2\neq \emptyset$. Assume that there exists $q\in U_1\cap U_2$ such that $\phi_1(q)\neq \phi_2(q)$,
$$
|\phi_1(q)-\phi_2(q)|=\e>0.
$$
Here $|\cdot|$ is the Euclidean norm in some local coordinates of $N$, whose domain of definition contains the images of some sufficiently small neighborhood of $q$ in both redefinitions of $\phi$. There is a neighborhood $U$ of $q$ such that
$$
|\phi_1(q)-\phi_1(p)|<\e/2 \mbox{ and } |\phi_2(q)-\phi_2(p)|<\e/2
$$
holds for all $p\in U$. Since $\phi_1$ and $\phi_2$ agree a.e., there exists $p'\in U$ with $\phi_1(p')=\phi_2(p')$. We have
$$
|\phi_1(q)-\phi_2(q)|\leq |\phi_1(q)-\phi_1(p')|+ |\phi_2(p')-\phi_2(q)|<\e,
$$
which is a contradiction.

If $\phi$ is non-constant and continuous, then its localizations are open and discrete; see e.g.~\cite{Rickman}. Therefore, $\phi$ itself is open and discrete. The Jacobian determinant of $\phi_{\a\b}$ is non-vanishing a.e. in this case~\cite[Thm. 16.10.1]{Iwaniec}. The chain rule holds by~\cite[Thm. 16.13.3]{Iwaniec}.

Let us denote by $\langle\cdot,\cdot\rangle$ the standard inner product of vectors in $\R^n$. We have
\begin{align*}
 |d(&u\circ\phi)|^2_g=\langle D\phi^Tdu|_\phi , g^{-1} D\phi^Tdu|_\phi\rangle =  \langle du|_\phi, h^{-1}|_\phi (h|_\phi D\phi g^{-1} D\phi^T) du|_\phi\rangle \\
  &\leq \tr{D\phi g^{-1} D\phi^T h|_\phi} |du|_h^2\circ\phi = \tr{g^{-1} \phi^*h} |du|_h^2\circ\phi  \\
  &\leq n^{n/2} K^{2/n} \Detg{\phi^*h}{g}^{1/n} |du|_h^2\circ\phi.
\end{align*}
In the second line we have used the fact that the matrix $h|_\phi D\phi g^{-1} D\phi^T$ has positive eigenvalues a.e. We have proven $(1)-(3)$.

$(4)$: First we have to show that the composition $\psi\circ\phi$ is localizable. Let $p\in M$. By the definition of the localizability there exist an open neighborhood $V$ of $\phi(p)$ and a coordinate neighborhood $(W,\{z^i\})$ containing $\psi(\phi(p))$ such that $\psi(V)\subset\subset W$. Since $\phi$ can be assumed continuous by part $(1)$, we have that $U:=\phi^{-1}(V)$ is open. Thus $\psi\circ\phi$ maps a neighborhood $U$ of $p$ into a coordinate neighborhood $W$ of $\psi(\phi(p))$ with $(\psi\circ\phi)(U)\subset\subset W$. That $\psi\circ\phi$ has locally constant sign follows from chain rule~\eqref{chain_rule}.

By the definition of the Sobolev space $\W(M,N)$, the component mapping $\psi^i=z^i\circ\psi$ is of Sobolev class $\W(V)$. By part $(2)$ of this theorem, the composition $\psi^i\circ\phi$ belongs to the Sobolev space $\W(U)$ and satisfies
\begin{equation*}
 \p_j(\psi^i\circ\phi)=\p_k\psi^i|_\phi\, \p_j\phi^k \mbox{ a.e.}
\end{equation*}
It follows that $\psi\circ\phi\in \W(M,L)$ and that the chain rule holds a.e. The pushforward and pullback formulas are consequences of the chain rule. 

The distortion inequality follows from
\begin{align*}
 K^2(g,(\psi\circ\phi)^*k)&\leq n^n K^2(g,\phi^*h)K^2(\phi^*h,\phi^*\psi^*k) \\
  &\leq n^n K^2 K^2(h,\psi^*k)|_\phi\leq n^nK^2 (K')^{2}
\end{align*}
holding a.e. Here we have applied the pullback formula~\eqref{pullback} with $T=k$, the inequality~\eqref{submult}, and the fact that the Jacobian determinant of $\phi$ is non-vanishing a.e. The latter implies that~\eqref{functoriality} holds a.e.

$(5)$: The inverse of $\phi$ is localizable as a continuous mapping. From the theory of quasiconformal mappings on $\R^n$ it follows that the inverse is of Sobolev class $\W(N,M)$; see e.g.~\cite[p.215]{Reshetnyak}.  We use Proposition~\ref{distortion_function} to calculate the distortion of $\phi^{-1}$:
\begin{align*}
 K^2(h,\phi^{-1*}g)&=\frac{\Trg{\phi^{-1*}g}{h}^n}{\Detg{\phi^{-1*}g}{h}}=\frac{\tr{h^{-1}D(\phi^{-1})^Tg|_\phi D(\phi^{-1})}}{\det{h^{-1}D(\phi^{-1})^Tg|_\phi D(\phi^{-1})}} \\
  &=\l.\frac{\tr{(D\phi^T h|_\phi D\phi)^{-1}g}}{\det{(D\phi^T h|_\phi D\phi)^{-1}g}}\r|_{\phi^{-1}}=K(\phi^*h,g)|_{\phi^{-1}} \\
  &\leq K^2(h,\phi^*g)^{n-1}|_{\phi^{-1}} \leq (K^2)^{n-1}.
\end{align*}
Here we have used the chain rule~\eqref{chain_rule} for the composition of a $\W$ mapping $\phi^{-1}$ and the quasiconformal mapping $\phi$ together with the fact that the Jacobian determinant of $\phi$ is non-vanishing a.e. to deduce that
\begin{equation*}
 (D\phi)^{-1}=(D\phi^{-1})|_\phi \mbox{ a.e.}
\end{equation*}
We conclude that $\phi^{-1}$ is quasiconformal.

For the integration by substitution formula we choose a partition of unity subordinate to $\{U_\a\}$ and apply~\cite[p.99]{Reshetnyak}. We have
\begin{align*}
 \int_M &f\circ \phi\, \Det{D\phi}d\mu_g=\sum_\a \int_{U_\a} \varphi_\a\, f\circ \phi\, |J_\phi|\, \det{h}|_\phi^{1/2} dx^n \\
  &= \sum_\a \int_{\phi(U_\a)} f \varphi_\a\circ \phi^{-1}\det{h}^{1/2}dy^n= \int_{N} f d\mu_h.
\end{align*}
\end{proof}


\subsection{A convergence theorem}\label{convsection}
We continue with a convergence theorem for Riemannian quasiregular mappings. For this, we first define a natural topology for Riemannian quasiregular mappings.

For quasiregular mappings on $\R^n$ the natural topologies are the weak Sobolev $\W$ topology and the compact-open topology. A sequence of $K$-quasiregular mappings converging in either of these topologies converge to a quasiregular mapping~\cite{Iwaniec,Rickman} with the same quasiregularity factor $K$. The analogue of (weak) $\W$ convergence of mappings in $W^{1,n}(M,N)$ we use is the following.
\begin{definition}
A sequence $\{\phi_i\}$ of Sobolev $W^{1,n}(M,N)$ mappings \emph{converge (weakly)} to $\phi:M\to N$ in $W^{1,n}(M,N)$ if $\phi$ is localizable and if for any localization $\phi:U\to V$ there is $N_U$ such that
$$
\phi_i(U)\subset\subset V
$$
for all $i\geq N_U$ and $\phi_i:U\to V$ converge (weakly) to $\phi:U\to V$ in $W^{1,n}(U,V)$. The convergence in $\W(M,N)$ is defined analogously.
\end{definition}
As the definition above suggests, to infer additional properties for the limit mapping $\phi$ we must a priori know that the sequence and the limit mapping can be localized simultaneously. Therefore we introduce the following concept. We call a sequence $\{\phi_i\}$ of mappings $M\to N$ \emph{uniformly localizable} with respect to a mapping $\phi:M\to N$ if $\phi$ is localizable and if for any localization $\phi:U\to V$ there is $N_U\in\N$ such that $\phi_i(U)\subset\subset V$ for $i\geq N_U$.

Let $d$ and $e$ be the distance metrics induced by the continuous Riemannian metrics $g$ and $h$ on $M$ and $N$. It is well known that for smooth Riemannian metrics the topologies induced by $d$ and $e$ coincide with the original topologies of the manifolds. The same is true also for continuous Riemannian metrics. This can be seen from the proof of this fact in the smooth case presented in~\cite[Thm. 1.18]{Aubin}. We have the following.
\begin{lemma}
 Let $\{\phi_i\}$ be a sequence of continuous mappings $(M,d)\to (N,e)$ converging uniformly on compact sets to $\phi:M\to N$. Then $\{\phi_i\}$ is uniformly localizable with respect to $\phi$. 

\end{lemma}
\begin{proof}
By the uniform convergence, the limit mapping $\phi$ is continuous and thus localizable. Let $p\in M$ and choose a neighborhood $U$ of $p$ and a coordinate neighborhood $V$ of $\phi(p)$ such that $\phi(U)\subset\subset V$. Choose $\d>0$ such that
\begin{equation*}
 e(\phi(x),\phi(p))\leq D/4
\end{equation*}
if $x\in \ol{B}(p,\d)$. Here $D$ is the distance $\mbox{dist}_e(\phi(p),\p V)$. Since $\phi_i\to \phi$ uniformly on $\ol{B}(p,\d)$, we can choose an integer $N$ such that
\begin{equation*}
 e(\phi_i(x),\phi(x))\leq D/4
\end{equation*}
for all $x\in\ol{B}(p,\d)$ and $i\geq N$. Now, let $x\in B(p,\d)$ and let $i\geq N$. We have
\begin{equation*}
 e(\phi_i(x),\phi(p))\leq e(\phi_i(x),\phi(x))+e(\phi(x),\phi(p))\leq D/2.
\end{equation*}
\end{proof}

As noted before, a sequence of $K$-quasiregular mappings on $\R^n$ converging weakly in $\W$ converges to a $K$-quasiregular mapping~\cite{Iwaniec}. It is also a fact that a sequence of quasiregular mappings on $\R^n$ converging uniformly on compact sets converges weakly on $\W$~\cite{Rickman}. These remarks together with the previous lemma allow us to prove a natural convergence theorem for Riemannian quasiregular mappings.

\begin{thm}\label{uniform_conv}
 Let a sequence $\{\phi_i\}$ of Riemannian $K$-quasiregular mappings $(M,g,d)\to (N,h,e)$ converge uniformly on compact sets to $\phi$. Then $\phi$ is Riemannian $K$-quasiregular.
\end{thm}
\begin{proof}
Since Riemannian quasiregular mappings are continuous, $\phi$ is localizable by the previous lemma. Let $p\in M$ and let $\phi:U\to V$ be a localization of $\phi$ at $p$. Applying the previous lemma again, we may assume without loss of generality that $\ol{U}$ is compact and that $\phi(U),\phi_i(U)\subset\subset V$.

Let us first show that $\phi:U\to V$ is of Sobolev class $\W(U,V)$. The coordinate representations of $\phi_i$ are $\ol{K}$-quasiregular mappings on $\R^n$ (by Thm.~\ref{localization}) converging uniformly on compact sets to the coordinate representation of $\phi:U\to V$. It follows from the proof of Theorem 8.6 in~\cite{Rickman} that the coordinate representations converge weakly in $W^{1,n}(U,V)$ to a $\ol{K}$-quasiregular mapping on $\R^n$. Thus $\phi\in W^{1,n}(U,V)$, and since $U$ was arbitrary, we have $\phi\in \W(M,N)$. 

By passing to a subsequence we may assume that either $J_{\phi_i}\geq 0$ or $J_{\phi_i}\leq 0$ a.e. on $U$ for all $i$. We may further assume the first condition, $J_{\phi_i}\geq 0$ a.e. on $U$ for all $i$, for the proof of the other case is analogous. Since the limit mappings $\phi:U\to V$ is quasiregular on $\R^n$ we have $J_\phi \geq 0$ a.e. on $U$. It follows that $J_\phi$ has locally constant sign.  

It remains to show that the distortion inequality
\begin{equation*}
 ||D\phi||^n\leq K\, \Det{D\phi}
\end{equation*}
holds a.e. for the localization $\phi:U\to V$. We prove this by modifying the proof of the analogous theorem for quasiregular mappings on $\R^n$~\cite[Thm. 8.7.1.]{Iwaniec}. We begin by showing the lower semicontinuity of the operator norm,
\begin{equation}\label{lowersemi}
\int_U ||D\phi||^n dx  \leq \liminf_{i} \int_U ||D\phi_i||^n dx.
\end{equation}
We use a modification of a standard argument~\cite[Thm. 2.11]{Lieb}. Let us first show that the weak convergence of $D\phi_i$ in $L^{n}(U)$ implies that $(h\circ\phi_i)^{1/2}D\phi_i$ converges weakly to $(h\circ\phi)^{1/2} D\phi$ in $L^n(U)$. 

Let $\psi\in L^{n}(U)^*$ and denote $k_i=(h\circ \phi_i)^{1/2}$ and $k=(h\circ \phi)^{1/2}$. We have
\begin{align*}
 \l| \int_U\r.&\l. \psi k_i D\phi_i-\int_U \psi k D\phi\r|\leq \l| \int_U \psi (k_i-k) D\phi_i+ \int_U \psi k (D\phi_i-D\phi) \r| \\
  &\leq \l| \int_U \psi (k_i-k) D\phi_i\r|+ \l|\int_U \psi k (D\phi_i-D\phi) \r| \\
  &\leq ||k_i-k||_{C(U)}\int_U |\psi| |D\phi_i | + \l|\int_U \psi k (D\phi_i-D\phi) \r| \\
  &\leq  ||k_i-k||_{C(U)} \l(\int_U |\psi|^{n/(n-1)}\r)^{\frac{n-1}{n}} \l(\int_U |D\phi_i |^n\r)^{\frac{1}{n}}  \\
  &+  \l|\int_U \psi k (D\phi_i-D\phi) \r|.
\end{align*}
Here the norms $|\cdot|$ are understood as the Hilbert-Schmidt norms where applicable. We have used H\"older's inequality in the ultimate inequality. Since $k_i\to k$ uniformly and the weakly convergent sequence $D\phi_i$ is bounded, we conclude that $k_iD\phi_i\to kD\phi$ weakly in $L^n(U)$. 

To simplify the following argument, let us now denote $h_i=h\circ\phi_i$ and, with a slight abuse of notation, let us also denote $h=h\circ\phi$. Consider the linear functional $L$ on the space of $L^n$ integrable $n\times n$ matrices $L^n(U,\R^{n\times n})$ defined by
\begin{equation*}
 L(T)=\int_U ||D\phi||^{n-2}\tr{g^{-1} D\phi^T h^{1/2} T}, 
\end{equation*}
where $T\in L^n(U,\R^{n\times n})$. Since $||D\phi||^{n-2}g^{-1}h^{1/2}\in L^{n/(n-1)}(U,\R^{n\times n})$, the dual of $L^n(U,\R^{n\times n})$, the functional $L$ is continuous by H\"older's inequality. By the weak convergence of $(h\circ\phi_i)^{1/2}D\phi_i$ we therefore have
\begin{align*}
 \int_U ||D\phi||^n&=\lim L(h_i^{1/2}D\phi_i)=\lim \int_U ||D\phi||^{n-2}\tr{g^{-1} D\phi^T h^{1/2} h_i^{1/2} D\phi_i} \\
  &\leq \liminf \int_U ||D\phi||^{n-2}\tr{g^{-1} D\phi^T h D\phi}^{1/2} \tr{g^{-1} D\phi_i^T h_i D\phi_i}^{1/2} \\
  & =\liminf \int_U ||D\phi||^{n-1} ||D\phi_i|| \\
  &\leq \l(\int_U ||D\phi||^{n}\r)^{\frac{n-1}{n}}  \liminf  \l(\int_U ||D\phi_i||^n\r)^{\frac{1}{n}}.
\end{align*}
Here in the second line we have used the Cauchy-Schwartz inequality for the Hilbert-Schmidt inner product of matrices and in the last line we have used H\"older's inequality. Diving by 
$$
\l(\int_U ||D\phi||^{n}\r)^{\frac{n-1}{n}}
$$
shows the lower semicontinuity of the operator norm~\eqref{lowersemi}. We also note for later use that the same argument applies if we multiply both of the integrands of the inequality~\eqref{lowersemi} by any positive test function.

We show next that, for any positive test function $\varphi\in C_c^\infty(U)=\mathcal{D}$, we have
\begin{equation}\label{J_conv}
\int_U \varphi\, \Det{D\phi_i}\to \int_U \varphi\, \Det{D\phi}.
\end{equation}
Since $J_{\phi_i}\geq 0$ a.e. on $U$, we have $\Det{D\phi_i}=\det{g^{-1}h|_{\phi_i}}^{1/2} J_{\phi_i}$. The sequence $\det{g^{-1}h|_{\phi_i}}\in C(U)$ converges to $\det{g^{-1}h|_\phi}$ uniformly on $U$. For simplicity, let us denote $a_i=\det{g^{-1}h|_{\phi_i}}^{1/2}$ and $a=\det{g^{-1}h|_{\phi}}^{1/2}$. We have $J_{\phi_i}\to J_{\phi}$ in $\mathcal{D}'$ by~\cite[Thm. 8.2.1]{Iwaniec}. It follows that 
$$
\int_U\varphi\,J_{\phi}, \int_U\varphi\,J_{\phi_i}< D
$$
for all $i$ large enough for some finite constant $D$. Let $\e>0$ and choose $a^*\in C_c^\infty(U)$ such that $||a^*-a||_{C(U)}< \e/D$. The standard approximation technique yields 
\begin{align*}
  \l|\int_U\r. &\l.\varphi \, \Det{D\phi_i} - \int_U \varphi \, \Det{D\phi} \r|=\l|\int_U \varphi\, a_i\,J_{\phi_i}-\int_U \varphi\, a\,J_{\phi}\r| \\
  &\leq \l|\int_U \varphi\,(a_i-a) J_{\phi_i} \r| + \l|\int_U \varphi\, (a-a^*) J_{\phi_i} \r|  +  \l|\int_U \varphi\,a^*(J_{\phi_i}-J_\phi)\r|  \\
  & +  \l|\int_U \varphi\, (a^*-a)J_\phi \r| \leq  ||a_i-a||_{C(U)} \int_U\varphi\,J_{\phi_i} + ||a-a^*||_{C(U)} \int_U\varphi\,J_{\phi_i} \\
  &+  \l|\int_U \varphi\,a^*(J_{\phi_i}-J_\phi)\r| + ||a^*-a||_{C(U)}\int_U\varphi\,J_{\phi}.          
\end{align*}
Taking the limit $i\to \infty$ shows that the right hand side is at most $2\e$. Since $\e$ was arbitrary, we have~\eqref{J_conv}.

Finally, the distortion inequality follows from~\eqref{lowersemi} and~\eqref{J_conv} since
\begin{align*}
  \int_U &\varphi\, ||D\phi||^n\leq \liminf \int_U \varphi\, ||D\phi_i||^n \\
  &\leq K\,\lim \int_U \varphi\, \Det{D\phi_i} = K \int_U \varphi\, \Det{D\phi}
\end{align*}
for all positive test functions $\varphi$.
\end{proof}

In this section we studied the basic properties of Riemannian quasiregular mappings. We proceed to study how Riemannian quasiregular mappings act on the space of conformal structures.

\chapter{An invariant conformal structure of a quasiconformal group}
We apply the Riemannian definition of quasiconformal mappings and the results of the previous sections. We study the action of a group of quasiconformal mappings on the space of Riemannian metrics on a manifold. The result we give is a generalization, in the countable case, of a result by Tukia~\cite[Thm. F]{Tukia}:
\begin{thmnonum}\label{tukia_thm}
 Let $\Gamma$ be a quasiconformal group of mappings $\phi:U\rightarrow U$, $U$ open in $\R^n \cup \{\infty\}$. Then there is a $\Gamma$-invariant conformal structure $G$ on $U$.
\end{thmnonum}

A quasiconformal group is a group of quasiconformal mappings that all have a same quasiconformality constant $K$. In particular, the distortion of all the iterates of the elements in the group remain bounded by that same $K$. 

A conformal structure on $\R^n$ is a measurable symmetric positive definite matrix field $G(x)\in \R^{n\times n}$ such that distortion function applied to the identity matrix, the Euclidean metric, and to $G$ is bounded:
$$
K^2(I,G(x))\leq K^2 \mbox{ a.e. }
$$
A conformal structure $G$ is an \emph{invariant conformal structure} of a quasiconformal group $\Gamma$ if all the mappings $\phi$ of the group $\Gamma$ satisfies a.e. the Beltrami equation:
\begin{equation*}
 D\phi(x)^TG(\phi(x))D\phi(x)=J_\phi(x)^{2/n}G(x),  \ \phi\in \Gamma.
\end{equation*}
The determinant of the conformal structure $G$ is normalized to $1$.

On a Riemannian manifold $(M,g)$, we define a \emph{conformal structure} to be a Riemannian metric $h$ whose $g$-invariant determinant equals $1$ a.e. and 
$$
K^2(g,h)\leq K^2 \mbox{ a.e. on } M.
$$ 
We do not assume other regularity assumptions on conformal structures other than measurability. Equivalently, a conformal structure can be understood as a measurable Riemannian metric whose volume form equals a.e. that of $g$ (cf.~\cite{Liimatainen1}). A conformal structure $h$ is an \emph{invariant conformal structure} for a quasiconformal group $\Gamma$ if every mapping of the group is conformal with respect to $h$:
\begin{equation*}
 \phi^*h=c\, h \mbox{ a.e. on } M, \  \phi\in \Gamma.
\end{equation*}
Here the conformal factor $c=c_\phi$ is positive a.e. and can be solved by taking the determinant of the equation:
\begin{equation*}
 c=\Detg{\phi^*h}{h}^{1/n}=\Detg{\phi^*g}{g}^{1/n}\equiv\Detg{D\phi}{g}^{2/n}.
\end{equation*}
Here the second equality holds by the assumption that the invariant determinant of $h$ equals $1$. We prove the following theorem that generalizes Tukia's result to general Riemannian manifolds in the countable case.
\begin{thm}\label{mainthm}
 Let $(M,g)$ be a Riemannian manifold with a continuous Riemannian metric $g$ and let $\Gamma$ be a countable quasiconformal group on $M$. Then there is a measurable Riemannian metric $h$ such that the equation
\begin{equation*}
 \phi^*h=c\,h
\end{equation*}
holds a.e. on $M$ for all mappings $\phi$ in the group $\Gamma$. Here $c$ depends on $\phi$. The volume form of $h$ equals that of $g$ a.e. and it holds that
$$
K^2(g,h)\leq H^2 \mbox{ a.e. on } M
$$
where $H$ depends only on $n$ and $K$.
\end{thm}
Here a quasiconformal group means a group of Riemannian quasiconformal mappings all having a same quasiconformality constant $K$.

We motivate the study of invariant conformal structures by the recent interest concerning invariant conformal structures on Riemannian manifolds and even in the sub-Riemannian setting~\cite{Pankka, Balogh}. See also~\cite{Martin, Kirsi, Astola} on the existence of uniformly quasiregular mappings (i.e. quasiregular semigroups generated by a single mapping) on various manifolds. To increase the number of potential applications of the theorem above, we discuss its generalization to abelian semigroups.

Conformal structures are sections of the bundle $\mathcal{S}$ of positive definite symmetric $2$-covariant tensors on $M$ whose $g$-invariant determinant is equal to one. The bundle $\mathcal{S}$ admits a special geometry that we employ in the proof of the above theorem.

\section{The bundle $\mathcal{S}$}\label{bundleS}
Let $(M,g)$ be a Riemannian manifold and $T^{2}_{0}M$ the bundle of its $2$-covari\-ant tensors, the tensors with two lower indices. We consider the subset $\mathcal{S}$ of $T^{2}_{0}M$ that consists of symmetric positive definite tensors $A$ whose invariant determinant
\begin{equation*}
 \Det{A}=\det{g^{-1}A}
\end{equation*}
equals one. We give $\mathcal{S}$ a fiber bundle structure over $M$, which will have the following properties. The fibers of $\mathcal{S}$ are naturally diffeomorphic to the manifold $\mathcal{P}$ of positive definite symmetric determinant one matrices. A transition function of the bundle $\mathcal{S}$ can be taken to be a mapping from an open subset of $M$ to the set of orthogonal matrices $O(n)$. That is, the structure group of $\mathcal{S}$ is $O(n)$. The construction of the fiber bundle here follows an accessible introduction to fiber bundles of~\cite{Frankel}.

Let us start by constructing local trivializations for $\mathcal{S}$. For this, choose a (continuous) orthonormal coframe $\{e^i\}$ on an open subset $U$ of $M$. We can write an element $A \in T^{2}_{0}U\cap \mathcal{S}$ as
\begin{equation*}
 A=A_{ij}e^i\otimes e^j.
\end{equation*}
Here the components $A_{ij}$ define a positive definite symmetric matrix. In an orthonormal coframe, it also holds that $\det{A}=\det{g^{-1}A}=\Det{A}=1$. Thus $A$ can be identified with a matrix of the manifold $\mathcal{P}$. A local trivialization is a mapping
\begin{equation}\label{loc_triv}
 \psi:\pi^{-1}U \rightarrow U\times \mathcal{P}, \hspace{10pt} A\mapsto (\pi(A),A_{ij}),
\end{equation}
where $\pi$ is the restriction of the bundle projection $T^2_0M\rightarrow M$ to $\mathcal{S}\rightarrow M$.
 
In case $(\tilde{\psi},V)$ is another trivialization corresponding to an orthonormal coframe $\{\tilde{e}^i\}$, we have a representation for $A\in \pi^{-1}(U\cap V)$ as
\begin{equation*}
  A=\tilde{A}_{ij}\tilde{e}^i\otimes \tilde{e}^j.
\end{equation*}
Since any two orthonormal basis are related by an orthogonal transformation, it follows that
\begin{equation*}
 A=\tilde{A}_{ij}\tilde{e}^i\otimes \tilde{e}^j=\tilde{A}_{ij} (h e^i)\otimes (h e^j)=\tilde{A}_{ij}h^i_ke^k\otimes h^j_le^l,
\end{equation*}
where $h=(h^i_j)\in O(n)$ is an orthogonal matrix. Therefore, the matrices $(A_{ij})$ and $(\tilde{A}_{ij})$ representing the same $A\in \mathcal{S}$ are related by
\begin{equation}\label{trans}
 A=h^T\tilde{A}h.
\end{equation}
A transition function $c_{UV}$ at $p\in U\cap V$ is a diffeomorphism of $\mathcal{P}$. By the calculation above, we have
\begin{equation*}
 c_{UV}(p)[A]=h(p)^TA h(p), 
\end{equation*}
where $h:U\cap V\rightarrow O(n)$. Thus $\mathcal{S}$ is an $O(n)$ fiber bundle over $(M,g)$ with fibers diffeomorphic to $\mathcal{P}$. We call $\mathcal{P}$ the \emph{typical fiber} of $\mathcal{S}$ and a fiber $\pi^{-1}(p)$ over $p\in M$ is denoted by $\mathcal{S}_p$.

The set $\mathcal{P}$ of positive definite symmetric matrices with determinant $1$ is a smooth manifold. It can be identified with 
\begin{equation*}
 SL(n)/SO(n).
\end{equation*}
We equip $\mathcal{P}$ with a Riemannian metric $g^\mathcal{P}$ defined by 
\begin{equation*}
 g^\mathcal{P}_A(X,Y):=\tr{A^{-1}XA^{-1}Y}, \hspace{10pt} A\in \mathcal{P}, \hspace{10pt} X,Y\in T_A \mathcal{P}.
\end{equation*}
The bilinear symmetric form $g^\mathcal{P}$ is indeed a Riemannian metric. The positive definiteness can be seen by writing 
\begin{equation*}
 g^\mathcal{P}_A(X,X)=\tr{A^{-1}XA^{-1}X}=||A^{-1/2}XA^{-1/2}||_{HS}^2.
\end{equation*}
Here $||\cdot||_{HS}$ is the usual Hilbert-Schmidt norm of matrices. The equality follows from the fact that positive definite symmetric matrices have symmetric square roots and the fact that tangent vectors of symmetric matrices are symmetric. 

The Riemannian metric $g^\mathcal{P}$ is invariant under the action
\begin{equation*}
 Z[A]=|\det{Z}|^{-2/n}Z^T A Z, \hspace{10pt} Z\in GL(n), \ A\in \mathcal{P}
\end{equation*}
 of the general linear group $GL(n)$.
Geometrically $(\mathcal{P},g^\mathcal{P})$ is a complete, simply connected, globally symmetric Riemannian manifold of negative curvature~\cite{Tukia, Helgason, Jostb}.

The geometry of $\mathcal{P}$ extends naturally to the fibers of $\mathcal{S}$. Let us first construct a fiber metric for $\mathcal{S}$. A fiber metric of a fiber bundle is a Riemannian metric for each fiber $\pi^{-1}(p)$. It is an inner product for tangent vectors $X,Y \in T_A \pi^{-1}(p)$, where $A\in \pi^{-1}(p)$. 

For the construction of the fiber metric, let $p\in M$, $A\in\pi^{-1}(p)$ and $X,Y\in T_A\pi^{-1}(p)$. By definition, tangent vectors $X$ and $Y$ of $T_A\pi^{-1}(p)$ are given by paths $A_1(t)$ and $A_2(t)$ through $A$ in $\pi^{-1}(p)$. In a local trivialization, we have
\begin{equation*}
 \begin{split}
  A_1(t)&=(\pi(\pi^{-1}(p)),A_{1ij}(t))=(p,A_{ij}+t\dot{A}_{1ij}(0)+\mathcal{O}(t^2)) \\
  A_2(t)&=(\pi(\pi^{-1}(p)),A_{2ij}(t))=(p,A_{ij}+t\dot{A}_{2ij}(0)+\mathcal{O}(t^2)). \\
 \end{split}
\end{equation*}
We define an inner product for $X$ and $Y$ using the local trivialization as
\begin{equation}\label{innerp}
\begin{split}
 \la X,Y\ra_A&:=\tr{A^{-1}\dot{A}_{1}(0)A^{-1}\dot{A}_2(0)}. \\
\end{split}
\end{equation}
It is the $g^\mathcal{P}$-inner product of the representation matrices. 
 
The inner product of~\eqref{innerp} is well defined. If $\tilde{A}, \tilde{A}_1(t)$ and $\tilde{A}_2(t)$ correspond to another trivialization, we have
\begin{equation*}
\begin{split}
 A^{-1}&=(h^T\tilde{A}h)^{-1}=h^{-1}\tilde{A}^{-1}h^{-T} \\
 \dot{A}_{1}(0)&=h^T\dot{\tilde{A}}_1(0)h \\
 \dot{A}_{2}(0)&=h^T\dot{\tilde{A}}_2(0)h \\
\end{split}
\end{equation*}
yielding
\begin{equation*}
\begin{split}
 \la X,Y\ra_A&=\tr{\l(h^{-1}\tilde{A}^{-1}h^{-T}\r)\l(h^T\dot{\tilde{A}}_1(0)h\r)\l(h^{-1}\tilde{A}^{-1}h^{-T}\r)\l(h^T\dot{\tilde{A}}_2(0)h\r)} \\
  &=\tr{\tilde{A}^{-1}\dot{\tilde{A}}_{1}(0)\tilde{A}^{-1}\dot{\tilde{A}}_2(0)}. \\
\end{split}
\end{equation*}
Thus the definition of the inner product calculated in different local trivializations of $\pi^{-1}U$ and $\pi^{-1}V$ agree in the overlap of $U$ and $V$. Fibers of $\mathcal{S}$ are isometric to the typical fiber $(\mathcal{P},g^\mathcal{P})$. 

We denote the fiber metric by $g^\mathcal{V}$, where $\mathcal{V}$ indicates that it is an inner product for the \emph{vertical vectors} of $\mathcal{S}$. The fiber distance $d^\mathcal{V}$ is induced by the fiber metric $g^\mathcal{V}$. In a local trivialization it holds, with a slight abuse of notation, that
\begin{equation}\label{distance_in_P}
 d^\mathcal{V}(A,B)=d^\mathcal{P}(A,B).
\end{equation}
Here $A$ and $B$ on the right hand side denote the corresponding representations of $A$ and $B$ in the local trivialization. On the left hand side $A$ and $B$ are elements of a fiber of $\mathcal{S}$.

We define the determinant normalized pullback $\phi_N^*$ of a (weakly) differentiable mapping $\phi:M\rightarrow M$ with a non-vanishing determinant. It is the pullback of $2$-covariant tensors whose determinant is normalized by the formula
\begin{equation*}
 \phi_N^*A:=\frac{\phi^*A}{\Detg{\phi^*g}{g}^{1/n}}.
\end{equation*}
The invariant determinant of a tensor $A\in \mathcal{S}_\phi$ is preserved:
\begin{equation*}
 \Detg{\phi_N^*A}{g}=\frac{\det{g^{-1}D\phi^TAD\phi}}{\det{g^{-1}D\phi^Tg|_\phi D\phi}}=\det{g^{-1}|_\phi A}=1.
\end{equation*}
For each $p\in M$, the normalized pullback is a mapping from the fiber of $\mathcal{S}$ over $\phi(p)$ to the fiber of $\mathcal{S}$ over $p$.

Finally, we check that the normalized pullback is an isometry between the fibers $\mathcal{S}_{\phi(p)}$ and $\mathcal{S}_p$ of $\mathcal{S}$ equipped with the fiber metric $g^\mathcal{V}$ for any $p\in M$. To simplify the notation, let us denote for a while $F= \phi_N^*:\mathcal{S}_{\phi(p)}\rightarrow \mathcal{S}_p$. With this notation, the isometry condition for $\phi_N^*$ reads
\begin{equation}\label{isom}
 g^\mathcal{V}_{F(A)}(F_* X,F_* Y)=g^\mathcal{V}_A(X,Y),
\end{equation}
where $F_*$ is the pushforward of $F$ and $X,Y\in T_A \mathcal{S}_{\phi(p)}$, $A\in \mathcal{S}_{\phi(p)}$. To verify this, choose paths $X(t)$ and $Y(t)$ tangents to $X$ and $Y$ at $A$ respectively. We have
\begin{equation*}
F_* X=\frac{d}{dt}|_{t=0}\phi_N^*X(t)=\frac{d}{dt}|_{t=0}\frac{D\phi^TX(t)D\phi}{\det{g^{-1}\phi^*g}^{1/n}}=\frac{D\phi^TXD\phi}{\det{g^{-1}\phi^*g}^{1/n}}.
\end{equation*}
and similarly for $Y$. This gives
\begin{equation*}
\begin{split}
 g^\mathcal{V}_{F(A)}(F_* X,F_* Y)&= \mbox{tr}\l(\l(\frac{D\phi^TAD\phi}{\det{g^{-1}\phi^*g}^{1/n}}\r)^{-1} \frac{D\phi^TXD\phi}{\det{g^{-1}\phi^*g}^{1/n}}\r.  \\
 &\times \l.\l(\frac{D\phi^TAD\phi}{\det{g^{-1}\phi^*g}^{1/n}}\r)^{-1} \frac{D\phi^TYD\phi}{\det{g^{-1}\phi^*g}^{1/n}}\r)  \\
 &=\tr{A^{-1}XA^{-1}Y}=g^\mathcal{V}_A(X,Y)
\end{split}
\end{equation*}
showing that the isometry condition holds.

\section{Proof of the theorem~\ref{mainthm}}
We prove Theorem~\ref{mainthm}. We study a \emph{set valued section} $E$ of $\mathcal{S}$ constructed from the normalized pullbacks of the mappings of the group $\Gamma$ applied to the Riemannian metric $g$ on $M$:
\begin{equation*}
 E(p)=\{\psi_N^*g(p): \psi\in \Gamma\}.
\end{equation*}
In case the initial metric $g$ happens to be an invariant conformal structure $E$ is just the initial metric $g$. Otherwise the quasiconformality condition will restrict the values of $E$ such that each set $E(p)\subset \mathcal{S}_p$ belongs to a unique smallest ball in the fiber metric of $\mathcal{S}$. The section of $\mathcal{S}$ consisting of the centers of the smallest balls is to be an invariant conformal structure. The proof is a generalization of the mentioned Tukia's result~\cite[Theorem F]{Tukia}.

\begin{proof}[Proof of Theorem~\ref{mainthm}]
Let $\Gamma$ be a countable $K$-quasiconformal group of homeomorphisms on a Riemannian manifold $(M,g)$. The countability of $\Gamma$ together with Theorem~\ref{rqr} imply that we can find a $\Gamma$-invariant measurable subset $N$ of $M$ which is of full measure and every $\phi\in \Gamma$ is differentiable on $N$ with a non-vanishing Jacobian determinant. By the properties of $N$, we can define a set valued section $E: N\rightarrow \mathcal{S}$ by
\begin{equation*}
 E(p)=\{\psi_N^*g(p): \psi\in \Gamma\}.
\end{equation*}

The Riemannian $K$-quasiconformality of the mappings in $\Gamma$ will give a uniform bound for the diameter of each set $E(p)\subset \mathcal{S}_p$. To see this, let $p\in N$ and $\phi\in \Gamma$ be arbitrary, and choose a local trivialization of $\mathcal{S}$ on a neighborhood of $p$ as in~\eqref{loc_triv}. We calculate the fiber distance between $\phi_N^*g$ and $g$ at $p$. In a local trivialization, we have by~\eqref{distance_in_P} that
\begin{equation}\label{fib_dist}
 d^\mathcal{V}_p(g(p),\phi_N^*g(p))=d^\mathcal{P}(g(p),\phi_N^*g(p)).
\end{equation}
On the right hand side we have the distance of the representation matrices in the typical fiber $\mathcal{P}$. 

In the local trivialization, the representation matrix of $g$ is just the identity matrix $I$. Accordingly, the representation matrix of $\phi^*g$ equals that of $g^{-1}\phi^*g$. Therefore the equation~\eqref{fib_dist} reads
\begin{equation*}
 d^\mathcal{V}_p(g(p),\phi_N^*g(p))=d^\mathcal{P}(I,g^{-1}(p)\phi_N^*g(p)).
\end{equation*}

It is shown in~\cite[p. 27]{Maass} that the distance of an element $A\in \mathcal{P}$ from the identity matrix $I$ is given by
 \begin{equation*}
 d^\mathcal{P}(I,A)=((\log{\mu_1})^2+\cdots +(\log{\mu_n})^2)^{1/2},
\end{equation*}
where $\mu_i$ are the eigenvalues of $A$. Let us denote by $\mu_i$ the eigenvalues of $g^{-1}(p)\phi_N^*g(p)$.  We observed in~\eqref{ratio_bound} that the ratio of the largest and smallest eigenvalue of $g^{-1}\phi^*g$ is bounded by $n^nK^2$. Of course, the same bound holds for all the ratios $\lambda_i/\lambda_j$ of the eigenvalues $\lambda_i$ of $g^{-1}\phi^*g$, $i,j=1,\ldots,n$. We have
\begin{equation*}
 \mu_i^n=\frac{\lambda_i^n}{\lambda_1\cdots\lambda_n}\leq (n^nK^2)^n
\end{equation*}
and it follows that
\begin{equation}\label{unifbd}
 d^\mathcal{V}_p(g(p),\phi_N^*g(p))=((\log{\mu_1})^2+\cdots +(\log{\mu_n})^2)^{1/2}\leq n^{1/2}\log{n^nK^2}.
\end{equation}
Since $\phi\in \Gamma$ and $p\in M$ were arbitrary, the diameter of each set $E(p)\subset \mathcal{S}_p$ is bounded by $2 n^{1/2}\log{n^nK^2}$ independently of $p$.
 
Let us then show that $E$ is invariant under the normalized pullbacks of the mappings in the group $\Gamma$. Normalized pullbacks satisfy a product rule
\begin{equation*}\label{prod_rule}
 \psi_N^*\circ \phi_N^*=(\phi\circ\psi)_N^*,
\end{equation*}
since the chain rule of differentiation~\eqref{pullback} holds and consequently
\begin{equation*}
 \Detg{\phi^*g}{g}|_\psi\Detg{\psi^*g}{g}=\Detg{(\phi\circ\psi)^*g}{g}.
\end{equation*}
By this product rule and the group property of $\Gamma$, we see that 
\begin{equation}\label{grpsol}
 (\phi_N^*E)(p)=\{\phi_N^*(\psi_N^*g)(p): \psi\in \Gamma\}=\{(\psi\circ\phi)_N^*g(p): \psi \in \Gamma\}=E(p),
\end{equation}
for all $\phi\in \Gamma$ showing the invariance of $E$. In this sense, $E$ can be viewed as a set valued solution to the problem of finding an invariant conformal structure for $\Gamma$.

It is shown in~\cite{Tukia} that each bounded set of $\mathcal{P}$ belongs to a unique smallest ball. Since the fibers of $\mathcal{S}$ are isometric to $\mathcal{P}$, it follows that each $E(p)$ belongs to a unique smallest ball. We denote by $h$ the section formed from the centers of the smallest balls.

We have seen that $E$ is invariant under normalized pullbacks of the group and recall from~\eqref{isom} that the normalized pullback is an isometry between the fibers of $\mathcal{S}$. Thus the unique smallest balls, and in particular their centers, are mapped to each other by normalized pullbacks. Accordingly,
\begin{equation*}
  \phi_N^*h(p)=h(p),
\end{equation*}
for every $p\in N$ and $\phi\in \Gamma$. That is
$$
\phi^*h =c\, h \mbox{ a.e. on }M
$$
with $d\mu_g$ equaling a.e. $d\mu_h$.  Here $c=\Detg{\phi^*g}{g}^{1/n}$. From~\eqref{unifbd} it also follows that $K^2(g,h)\leq H$ a.e. with $H$ depending only on $K$ and $n$.

It remains to show that $h$ is measurable. This follows from exactly the same argument as in the proof of Tukia's theorem~\cite[Theorem F]{Tukia}. Let $\Gamma=\{\phi_0, \phi_1,\ldots\}$ and approximate $E$ by $E_j(p)=\{\phi_{iN}^*g: i\leq j\}$. Consider $h_j$ to be the image of the mapping $H$, which maps $E_j$ to the section constructed from the centers of unique smallest balls. Since $H$ is continuous in the Hausdorff metric, it follows that each $h_j$ is measurable and that $h_j\rightarrow h$ as $j\rightarrow\infty$. Hence $h$ is measurable.
\end{proof}

We conclude this chapter by discussing how to prove the statement of Theorem~\ref{mainthm} in the case the group $\Gamma$ is replaced by an abelian quasiregular semigroup, by which we mean an abelian semigroup of Riemannian quasiregular mappings whose distortion is bounded by a same constant $K$. We see form the proof above that, apart from~\eqref{grpsol}, all the arguments and equations there remain valid for a semigroup of Riemannian $K$-quasiregular mappings even if the semigroup is not abelian. In~\eqref{grpsol} the group property of $\Gamma$ was used to deduce that $\Gamma\circ\phi$ equals $\Gamma$ for any $\phi\in\Gamma$. 

This difference between quasiconformal groups and quasiregular semigroups on this matter was addressed in the work by Iwaniec and Martin~\cite{Martin}. See also~\cite{Iwaniec}. They proved that any abelian quasiregular semigroup $\tilde{\Gamma}$ on an $n$-sphere admits an invariant conformal structure. The argument they used relied on the fact that quasiregular mappings on an $n$-sphere are open and discrete. Therefore, at every point on the sphere, every mapping of $\tilde{\Gamma}$ has only a finite collection of local inverses. This fact together with the fact that the inverse of a quasiconformal mapping is quasiconformal and the assumption that $\tilde{\Gamma}$ is abelian let them to establish an equation analogous to~\eqref{grpsol}.

As Riemannian quasiregular mappings are open and discrete, and the inverse of a Riemannian quasiconformal mapping is Riemannian quasiconformal, the argument explained above applies for an abelian semigroup of Riemannian $K$-quasiregular mappings on a compact manifold. In the non-compact case, at each point on the manifold, every mapping has at most countably many local inverses and the argument by Iwaniec and Martin applies also in that case. We have argued that the statement of Theorem~\ref{mainthm} continues to hold for any abelian semigroup of Riemannian quasiregular mappings on a Riemannian manifold.

\chapter{Concluding remarks}
This chapter reviews the three publications that form the main contributions of this thesis together with the theory of Riemannian quasiregular mappings developed in the previous chapter. In each review, a short summary of the publication is given first, which is followed by a more detailed presentation of the publication. The three publications studies the questions on the existence of conformal mappings, existence of optimal Riemannian metrics and the regularity of conformal mappings. The emphasis is on the global aspects of these questions and new tools to study the question are introduced in these publications.

Publication I is a non-existence result. In that publication it is proven that generic Riemannian manifolds of dimension $n\geq 3$ do not have any conformal symmetries. Publication II can be viewed as a functional analytic generalization of Tukia's result for mappings on Riemannian manifolds that preserve a given volume. This publication also proves a nontrivial generalization of Neumann's mean ergodic theorem and a fixed point theorem for nonlinear mappings on certain nonlinear spaces. The last publication concerns the regularity of conformal mappings via a system of $n$-harmonic coordinates that can be seen as a generalization of both the harmonic and of the isothermal coordinates.  The existence of $n$-harmonic coordinates is proven in this paper for Riemannian manifolds with $C^r$ regular, $r>1$, metric tensors.
\section{Publication I}
In this publication it is proven that generic smooth Riemannian manifolds of dimension $\geq 3$ do not admit any nontrivial local conformal diffeomorphisms. Consequently, generic manifolds of dimension $\geq 3$ do not admit nontrivial conformal Killing vector fields near any point. A direct application of this result to the inverse problem of Calder\'on on manifolds, which shows that generic manifolds of dimension $\geq 3$ do not admit limiting Carleman weights near any point of the manifold, is presented.

A subset of a topological space is \emph{residual} if it contains a countable intersection of open dense sets and we call a property \emph{generic} if it holds for all elements in some residual set. In this publication, the set of all smooth Riemannian metrics is equipped with the $C^\infty$ topology~\cite{BaUr} in which residual sets are dense. This publication gives a rigorous proof of the following theorem.
\begin{thm}
Let $M$ be a compact boundaryless $C^{\infty}$ manifold having dimension $n \geq 3$. There is a residual set of Riemannian metrics on $M$ for which there are no conformal diffeomorphisms between any distinct open subsets of the manifold. 
\end{thm}
In terms of the distortion function the result states that $K^2(g,\phi^*g)$ on a generic Riemannian manifold $(M,g)$ can not be made to equal one by any diffeomorphims $\phi$ even locally. This interpretation leaves open an interesting question, which is not present in the Euclidean case, on the infimum of the distortion in the set local diffeomorphisms. The proof of this theorem involves a dimension count argument based on jet spaces and a study of the convergence properties of conformal mappings. 

The result is a conformal analogue of a result by Sunada concerning nonexistence of local isometries on manifolds~\cite{Su}. The result makes precise the principle that generic Riemannian manifolds do not admit any conformal symmetries. The dimension two is special due to the existence of isothermal coordinates: a composition of isothermal coordinates with any restriction of a M\"obius transformation on the complex plane $\mathbb{C}$ gives conformal diffeomorphisms between distinct open sets on two dimensional Riemannian manifolds.

The result can be regarded as a partial inverse to Tukia's result~\cite{Tukia} and its generalization, Theorem~\ref{main_thm}. The Riemannian metric given by Tukia's result can be considered as an optimal Riemannian metric for the mappings in a group of quasiconformal mappings, whereas the generic Riemannian metrics described by the result are not optimal for any diffeomorphism on the manifold. The analogue of the result as an inverse to the result of Tukia is only partial due to the smoothness assumptions used in this publication. In particularly, it cannot be deduced that for a generic Riemannian manifold there are no quasiconformal groups.

Recall that a smooth vector field $X$ on a Riemannian manifold $(U,g)$ is called \emph{conformal Killing vector field} if the local flows generated by $X$ are conformal transformations.
Equivalently, $X$ is a conformal Killing vector field if and only if it satisfies
$$
\mathcal{L}_Xg-\Trg{\mathcal{L}_Xg}{g}g=0.
$$
The sufficiency part of this statement can be seen from Example~\ref{conformal_flow}. Since the local flows generated by conformal Killing vector fields are conformal diffeomorphisms, the existence of a nontrivial conformal Killing field near a point implies the existence of a conformal diffeomorphism from some open set $(U,g)$ onto a disjoint open set $(V,g)$. Therefore, the theorem has the following consequence.
\begin{cor} \label{cor_main2}
Let $M$ be a compact boundaryless $C^{\infty}$ manifold having dimension $n \geq 3$. There is a residual set of Riemannian metrics on $M$ which do not admit a nontrivial conformal Killing vector field near any point.
\end{cor}
A Riemannian manifold whose Riemannian metric is in the residual set described by this theorem is called \emph{nowhere conformally homogeneous} in this publication. 

Limiting Carleman weights were introduced in~\cite{KSU} as a general method for studying the inverse conductivity problem of Calder\'on in Eu\-clidean space. Subsequently, in \cite{DKSaU} it was shown that these weights are also useful in the Calder\'on problem for certain anisotropic conductivities and that article also gave a characterization of manifolds admitting limiting Carleman weights. Let $(U,g)$ be an open Riemannian manifold (that is, $U$ has no boundary and no component is compact). In this case $(U,g)$ admits a limiting Carleman weight if and only if $(U,cg)$ admits a parallel unit vector field for some positive function $c$~\cite{DKSaU}. A vector field is parallel on $(U,cg)$ if it satisfies 
$$
\nabla_{cg} X=0.
$$
Here $\nabla_{cg}$ is the Levi-Civita connection of the Riemannian metric $cg$. A parallel vector field with respect $\nabla_{cg}$ is a conformal Killing vector field with respect to the Levi-Civita connection $\nabla_g$ of the Riemannian metric $g$. This leads to the final corollary of this publication.
\begin{cor}
Let $(U,g)$ be an open submanifold of some compact manifold $(M,g)$ without boundary, having dimension $n \geq 3$. There is a residual set of Riemannian metrics on $M$ which do not admit limiting Carleman weights near any point of $U$.
\end{cor}

\section{Publication II}
The final main theorem of this publication gives a natural condition for when a mapping on a Riemannian manifold $(M,g)$ preserving the Riemannian volume form $d\mu$ is actually an isometry with respect to some other Riemannian metric $h$. The governing principle behind this theorem is analogous to the one in Theorem~\ref{mainthm}, but the framework and the statement of this theorem is more functional analytic in nature. Nontrivial generalizations of Neumann's mean ergodic theorem and a fixed point theorem to certain negatively curved metric spaces are the other main result of this publication. 

Let $(M,g)$ be a smooth, closed and oriented finite dimensional Riemannian manifold. Then the set of all (sufficiently smooth) Riemannian metrics $\M$ on $M$ can be considered as an infinite dimensional manifold whose tangent vectors can be given an $L^2$ inner product~\cite{Ebin, Clarke}. This publication considers the submanifold $\M_\mu$ of $\M$ consisting of Riemannian metrics that have the same volume form $d\mu$ induced by $g$. Equivalently, this infinite dimensional manifold consists of all the sections of the bundle $\mathcal{S}$ constructed in Chapter~\ref{bundleS} with respect to $g$.

The theory of infinite dimensional manifolds combines functional analysis and geometry. The tangent spaces $T_h\M_\mu$, $h\in \M_\mu$, consist of $g$-traceless $(0,2)$-tensor fields and the $L^2$-inner product for the tangent vectors is given by
$$
\langle U,V \rangle_g=\int_M \tr{g^{-1}Ug^{-1}V}d\mu.
$$
Under this inner product $\M_\mu$ can be considered as an infinite dimensional globally symmetric space of nonpositive curvature~\cite{Ebin, Freed, Clarke}. The induced Riemannian distance is given by the formula
\begin{equation}\label{ddist}
d^2(g,h)=\int_M\tr{(\log{(g^{-1}h)})^2}d\mu.
\end{equation}
The fact that the square of the distance $d$ is the integral of the square of the fiber distance $d^\mathcal{V}$, constructed in Chapter~\ref{bundleS}, gives $(\M_\mu, d)$ an interpretation as a curved $L^2$ space. The details of the used infinite dimensional setup are covered in the publication. 

Diffeomorphisms of $M$ preserving a given volume form are called volumorphisms. Volumorphisms act isometrically on $\M_\mu$ by pullback. A volumorphism $\phi$ preserves the volume of the tangent spheres of $M$, and, on a compact manifold the volumorphism $\phi$ is Riemannian quasiconformal (the distortion function $K^2(g,\phi^*g)$ is bounded). However, in general there is no restriction on the distortion of the iterates of the mapping. The following question is answered affirmative in this publication: ``If the action of a volumorphisms on $\M_\mu$ has a bounded orbit in the distance metric $d$, is there a fixed point of this action?''. 

In the case a fixed point exists, it belongs to a complete metric space $(X,\d)$ that contains $(\M_\mu,d)$ as an isometrically embedded subset. The metric space $(X,\d)$ consists of a.e. positive definite symmetric $(0,2)$-tensor fields whose volume forms agree with $d\mu$ a.e. The elements of $X$ have also the property that the integral in~\eqref{ddist} for any two elements $g$ and $h$ of $X$ is finite. Again, this condition should be considered as a curved $L^2$ condition for the elements of $X$. As a metric space, $(X,\d)$ is a complete global Alexandrov nonpositive curvature (NPC) space. See~\cite{Josta, Jostb} for details about Alexandrov NPC spaces. The first main theorem of this publication is the following. Here $\mathcal{D}_\mu^{s+1}$ denotes the space of volumorphisms of Sobolev class $H^{s+1}(M,M)$.
\begin{thm}
If the action of a volumorphism $\phi\in\mathcal{D}_\mu^{s+1}$ has a bounded orbit in $(X,\d)$ for some $p\in X$, then there exists a fixed point $g$ for the action in the $\d$-closure of the subset $co\{p,Tp,T^2p,\ldots\}$ of $X$. With respect to this fixed point, $\phi$ is a Riemannian isometry,
\begin{equation*}
 \phi^*g=g.
\end{equation*}
\end{thm}

To find a fixed point for the action of a volumorphism on $\M_\mu$, Neumann's mean ergodic theorem is generalized to suit the nonlinear setting of the publication. Mean ergodic theorems consider the convergence of averages of the iterates of the points under the action. In a general nonlinear setting there is no obvious notion of average, but on nonpositively curved spaces such a generalization of averages exists~\cite{Josta, Jostb, Karcher}. The mean ergodic theorem proven in this publications is the following.
\begin{thm}[Mean Ergodic Theorem]\label{main_thm} Let $(\mathcal{N},d)$ be a complete global Alexandrov NPC space and $T:\mathcal{N}\to\mathcal{N}$ a nonexpansive distance convex mapping. Then, for any $p\in \mathcal{N}$ whose orbit is bounded, and any $q\in \mathcal{N}$, the following are equivalent:
\begin{align*}
  (i) \qquad &Tq=q \mbox{ and } q\in \overline{co}\{p,Tp,T^2p,\ldots\}, \\
 (ii) \qquad  &q=\lim_{n}m_n(p), \\
 (iii) \qquad &q=\mbox{w-}\lim_n m_n(p), \\
  (iv) \qquad &q \mbox{ is a weak cluster point of the sequence } (m_n(p)).
\end{align*}
\end{thm}
Here $m_n(p)$ denote the average of the iterates of a point $p$. The notation $\mbox{w-}\lim$ and the notion of weak cluster point refers to generalizations of the definitions of weak limit and weak cluster point on Hilbert spaces to complete global Alexandrov NPC spaces~\cite{Josta}. A mapping is nonexpansive on $\mathcal{N}$ if 
$$
d(Tp,Tq)\leq d(p,q)
$$
holds for any $p,q\in \mathcal{N}$. The condition of distance convexity is a definition introduced in the publication that emerges naturally from the setting of the publication. A fixed point theorem without the additional assumption of distance convexity is also given in this publication.
\begin{thm}[Fixed point theorem]\label{fp_thm} 
Let $(\mathcal{N},d)$ be a complete global Alexandrov NPC space and $T:\mathcal{N}\to\mathcal{N}$ a nonexpansive mapping. Then, for any $p\in \mathcal{N}$ whose orbit is bounded, there exists a fixed point $q$ of $T$ in the subset $\ol{co}\{p,Tp,T^2p,\ldots\}$ of $\mathcal{N}$. 
\end{thm}

The mean ergodic theorem and the fixed point theorem are the other main results of this publication. These results are general and it is interesting to see whether they have other applications besides the one considered in this work in both finite and infinite dimensional settings.

\section{Publication III}
The third publication of this thesis studies the smoothness of conformal mappings between Riemannian manifolds whose metric tensors have limited regularity. It is shown that any bi-Lipschitz mapping, and more generally a Riemannian $1$-quasiregular mapping, between two manifolds with $C^r$ metric tensors ($r > 1$) is a $C^{r+1}$ conformal local diffeomorphism. This gives a new proof of a regularity result of conformal mappings by Iwaniec~\cite{Iwaniec_thesis}. The proof is based on $n$-harmonic coordinates, a generalization of both the standard harmonic coordinates and of the isothermal coordinates, that is particularly suited to studying conformal mappings. The existence of $n$-harmonic coordinates on general Riemannian manifolds is the other main result of this publication. A convergence theorem of Chapter~\ref{convsection} is used to prove that a sequence of conformal mappings between $C^r$ ($r>1$) Riemannian manifolds converging uniformly on compact sets converge to a $C^{r+1}$ conformal mapping. 

This article addresses the following question: ``Given a conformal mapping between two smooth ($=C^{\infty}$) manifolds having metric tensors of limited regularity, how regular is the mapping?''. If one considers isometries instead of conformal mappings, it is known that a distance preserving homeomorphism between two Riemannian manifolds with $C^r$ ($r > 0$) metric tensors is in fact a $C^{r+1}$ isometry~\cite{Myers, HW,  Calabi}. A recent article by Taylor~\cite{Taylor} gave a proof of this fact based on the systematic use of harmonic coordinates.

The modern Liouville theorems for conformal Sobolev $\W$ mappings on $\R^n$ show that mappings satisfying Beltrami equation
$$
D\phi^TD\phi=J_\phi^{2/n}
$$
in the weak sense (cf. Chapter~\ref{section_riemannian_def}) are restrictions of M\"obius transformations. The first part of the proof of the Liouville theorem under the weak regularity assumption consists of showing that mappings satisfying the Beltrami equation are regular enough for one to apply the classical proof of the Liouville theorem~\cite[Ch. 5]{Iwaniec}. In this publication, this part of the Liouville's theorem is generalized to Riemannian $1$-quasiregular mappings $\phi:(M,g)\to (N,h)$ that consequentially satisfy
\begin{equation}\label{conf1again}
\phi^*h=cg
\end{equation}
in the weak sense (Prop.~\ref{1qc}). The regularity assumptions on the Riemannian metrics in this publication are that they are H\"older $C^r$ regular with H\"older exponent $r>1$. In this case it is shown that a $\W(M,N)$ mapping $\phi$ satisfying~\eqref{conf1again} is a local $C^{r+1}$ conformal diffeomorphisms. Similarly as in~\cite{Taylor}, this is done via a special coordinate system in which tensors on the manifold have maximal regularity. However, instead of harmonic coordinates that are useful in the regularity analysis of isometries, a system of \emph{$n$-harmonic coordinates} is employed in this publication. 

A function $u$ on a Riemannian manifold is called $p$-harmonic ($1 < p < \infty$) if it satisfies the nonlinear elliptic equation ($p$-Laplace equation) 
$$
\delta(\abs{du}^{p-2} du) = 0.
$$
Here $\abs{\,\cdot\,}$ is the norm induced by the Riemannian metric, $d$ is the exterior derivative, and $\delta$ is the codifferential (the adjoint of $d$ in the $L^2$ inner product on differential forms). A coordinate system is called $p$-harmonic if each coordinate function is $p$-harmonic and the coordinate system is $n$-harmonic if $p$ equals the dimension $n$ of the manifold. The existence of $p$-harmonic coordinate systems on any Riemannian manifold with $C^r$ regular ($r>1$) metric tensor is the other main result in this paper.

The special property of $n$-harmonic functions is that pullbacks by conformal mappings preserve this class. In this way, any conformal mapping $\phi:(M,g)\to (N,h)$ can be locally expressed as a composition of an $n$-harmonic coordinate chart and the inverse of another such chart: Let $p\in M$ and $v=(v_1,\ldots,v_n)$ be $n$-harmonic coordinates around $\phi(p)\in N$. Then $u=v\circ\phi=(u_1,\ldots, u_n)$ is an $n$-tuple of $n$-harmonic functions. Any system of $n$-harmonic coordinates on a manifold with $C^r$ metric tensor ($r > 1$) is in this publication shown to have $C^{r+1}_*$ regularity. Here $C^{r+1}_*$ denotes the Zygmund space that is well known to coincide with the H\"older space $C^{r+1}$ for non-integer $r$. Since we have
$$
\phi=v^{-1}\circ u,
$$
the regularity of conformal mappings follows directly from these facts. 

Another special property of $n$-harmonic coordinates observed in the publication is that any isothermal coordinate system is necessarily an $n$-harmonic coordinate system. Together with the conformal invariance of $n$-harmonic coordinates stated above, $n$-harmonic coordinates can be regarded as a generalizations of isothermal coordinates. The $n$-harmonic coordinates seem like a natural tool, and it is an interesting question whether they can be used for other questions in conformal geometry besides the one treated in this paper. The present work was motivated by Publication I, where regularity of conformal mappings was a key point.

\providecommand{\bysame}{\leavevmode\hbox to3em{\hrulefill}\thinspace}
\providecommand{\href}[2]{#2}
\bibliographystyle{plain}
\def\cprime{$'$} \def\cprime{$'$} \def\cprime{$'$}




\addpublication{Tony Liimatainen and Mikko Salo}{Nowhere conformally homogeneous manifolds and limiting Carleman weights}{Inverse Problems and Imaging}{6, Issue 3, 523--530}{August}{2012}{American Institute of Mathematical Sciences}{j1}
\addcontribution{The results in this article were obtained in collaboration with the coauthor. Of the results, the author's main contribution was in the proof of the main theorem and its basic idea.  The article was written in close collaboration. The author is mainly responsible for the publication process of the article.}

\addpublication{Tony Liimatainen}{Optimal Riemannian metric for a volumorphism and a mean ergodic theorem in complete global Alexandrov nonpositively curved spaces}{AMS Contemporary Mathematics: Analysis, Geometry and Quantum field theory}{584, 163--178}{December}{2012}{American Mathematical Society}{j2}
\addcontribution{This article represents independent research of the author.}

\addpublication[submitted]{Tony Liimatainen and Mikko Salo}{$n$-harmonic coordinates and the regularity of conformal mappings}{Submitted to Mathematical Research Letters, 20 pages, preprint: arXiv:1209.1285}{}{September}{2012}{No copyright holder at this moment}{j3}
\addcontribution{The results in this article were obtained in collaboration with the coauthor. The author is responsible for the basic idea of the article. The article was written in close collaboration.}



\end{document}